\documentclass[12pt,a4paper]{amsart}
\usepackage{amssymb}
\usepackage[all,cmtip]{xy}
\usepackage{hyperref}

\calclayout

\newcommand{\+}{\protect\nobreakdash-}
\renewcommand{\:}{\colon}

\newcommand{\rarrow}{\longrightarrow}

\newcommand{\down}{\downarrow}

\newcommand{\lrarrow}{\mskip.5\thinmuskip\relbar\joinrel\relbar\joinrel
 \rightarrow\mskip.5\thinmuskip\relax}

\DeclareMathOperator{\Hom}{Hom}
\DeclareMathOperator{\coker}{coker}

\newcommand{\sA}{\mathsf A}
\newcommand{\sB}{\mathsf B}
\newcommand{\sC}{\mathsf C}
\newcommand{\sE}{\mathsf E}
\newcommand{\sM}{\mathsf M}
\newcommand{\sP}{\mathsf P}
\newcommand{\sS}{\mathsf S}
\newcommand{\sT}{\mathsf T}

\newcommand{\cM}{\mathcal M}
\newcommand{\cN}{\mathcal N}
\newcommand{\cP}{\mathcal P}
\newcommand{\cQ}{\mathcal Q}

\renewcommand{\k}{\Bbbk}

\newcommand{\Modl}{{\operatorname{\mathsf{--Mod}}}}
\newcommand{\Vect}{{\operatorname{\mathsf{--Vect}}}}
\newcommand{\vect}{{\operatorname{\mathsf{--vect}}}}

\newcommand{\inj}{\mathsf{inj}}
\newcommand{\proj}{\mathsf{proj}}
\renewcommand{\flat}{\mathsf{flat}}

\newcommand{\id}{\mathrm{id}}
\newcommand{\sop}{\mathsf{op}}

\newcommand{\Sets}{\mathsf{Sets}}

\newcommand{\Section}[1]{\bigskip\section{#1}\medskip}
\theoremstyle{plain}
\newtheorem{thm}{Theorem}[section]

\newtheorem{lem}[thm]{Lemma}
\newtheorem{prop}[thm]{Proposition}
\newtheorem{cor}[thm]{Corollary}
\theoremstyle{definition}
\newtheorem{ex}[thm]{Example}
\newtheorem{exs}[thm]{Examples}
\newtheorem{rem}[thm]{Remark}

\newtheorem{defn}[thm]{Definition}

\begin{document}

\author{Leonid Positselski}

\address{Institute of Mathematics, Czech Academy of Sciences \\
\v Zitn\'a~25, 115~67 Praha~1 \\ Czech Republic}

\email{positselski@math.cas.cz}

\title{On pure monomorphisms and pure epimorphisms \\
in accessible categories}

\begin{abstract}
 In all $\kappa$\+accessible additive categories, $\kappa$\+pure
monomorphisms and $\kappa$\+pure epimorphisms are well-behaved, as
shown in our previous paper~\cite{Plce}.
 This is known to be not always true in $\kappa$\+accessible
nonadditive categories.
 Nevertheless, mild assumptions on a $\kappa$\+accessible category are
sufficient to prove good properties of $\kappa$\+pure monomorphisms
and $\kappa$\+pure epimorphisms.
 In particular, in a $\kappa$\+accessible category with finite
products, all $\kappa$\+pure monomorphisms are $\kappa$\+directed
colimits of split monomorphisms, while in a $\kappa$\+accessible
category with finite coproducts, all $\kappa$\+pure epimorphisms
are $\kappa$\+directed colimits of split epimorphisms.
 We also discuss what we call Quillen exact classes of monomorphisms
and epimorphisms, generalizing the additive concept of one-sided
exact category.
\end{abstract}

\maketitle

\tableofcontents

\section*{Introduction}
\medskip

 The concept of $\kappa$\+purity is an important technical tool in
locally presentable and accessible category
theory~\cite[Section~2.D]{AR}.
 The properties of $\kappa$\+pure monomorphisms and $\kappa$\+pure
epimorphisms have been studied by Ad\'amek, Rosick\'y, and collaborators
in the papers~\cite{AHT,AR2,AR3}, which offer an assortment of theorems
and counterexamples (see also the paper by Hu and Pelletier~\cite{HP}
and the recent prepint of Kanalas~\cite{Kan}).
 However, the basic results about $\kappa$\+purity in
$\kappa$\+accessible categories seem to only have been proved under
assumptions that are somewhat restrictive.
 In this paper, we explain how to relax some assumptions.

 Let us list what we consider the basic results.
 In any $\kappa$\+accessible category with pushouts, all $\kappa$\+pure
monomorphisms are $\kappa$\+directed colimits of split
monomorphisms~\cite[Corollary and Remark~2.30]{AR}.
 In any $\kappa$\+accessible category with pushouts, all $\kappa$\+pure
monomorphisms are regular monomorphisms~\cite[Corollary~1]{AHT}.
 In any $\kappa$\+accessible category with pushouts, the class of
$\kappa$\+pure monomorphisms is stable under
pushouts~\cite[Corollary~2]{AHT}, \cite[Proposition~15(i)]{AR2}.

 In any $\kappa$\+accessible category with pullbacks, all
$\kappa$\+pure epimorphisms are $\kappa$\+directed colimits of split
epimorphisms~\cite[Proposition~3]{AR2}.
 In any $\kappa$\+accessible category with pullbacks, all
$\kappa$\+pure epimorphisms are regular
epimorphisms~\cite[Proposition~4(b)]{AR2}.
 In any locally $\kappa$\+presentable category, the class of
$\kappa$\+pure epimorphisms is stable under
pullbacks~\cite[Proposition~15(ii)]{AR2}.

 In a slightly different setting of the paper~\cite{AR3}, some results
similar to the above ones are stated under milder assumptions.
 In particular, according to~\cite[Lemma~2.2]{AR3}, existence of
weak pushouts in the full subcategory of finitely presentable objects
is sufficient for the pure monomorphisms in a finitely accessible
category to be directed colimits of split monomorphisms.
 By~\cite[Lemma~3.1]{AR3}, existence of weak pullbacks in the full
subcategory of finitely presentable objects is sufficient for the pure
epimorphisms in a finitely accessible category to be directed colimits
of split epimorphisms.
 
 For comparison, in any $\kappa$\+accessible additive category $\sA$,
all $\kappa$\+pure monomorphisms are $\kappa$\+directed colimits of
split monomorphisms, and all $\kappa$\+pure epimorphisms are
$\kappa$\+directed colimits of split epimorphisms.
 All $\kappa$\+pure monomorphisms in $\sA$ are regular monomorphisms,
and all $\kappa$\+pure epimorphisms are regular epimorphisms.
 All pushouts of $\kappa$\+pure monomorphisms always exist in $\sA$,
and the class of $\kappa$\+pure monomorphisms is stable under
pushouts.
 All pullbacks of $\kappa$\+pure epimorphisms exist in $\sA$, and
the class of $\kappa$\+pure epimorphisms is stable under pullbacks.

 Moreover, in any $\kappa$\+accessible additive category, all
$\kappa$\+pure epimorphisms have kernels, and all $\kappa$\+pure
monomorphisms have cokernels.
 The $\kappa$\+pure monomorphisms are precisely the kernels of
the $\kappa$\+pure epimorphisms, and the $\kappa$\+pure epimorphisms
are precisely the cokernels of the $\kappa$\+pure monomorphisms
(this is a generalization of~\cite[Proposition~5]{AR2}).
 All results mentioned in this and the previous paragraph follow from
the exposition in~\cite[Section~4]{Plce}, particularly from
the existence of the $\kappa$\+pure exact structure (in the sense
of Quillen) together with~\cite[Propositions~4.2 and~4.4]{Plce}.

 From our perspective, even such assumptions as existence of weak
pushouts and weak pullbacks are too restrictive, and unnecessarily so,
for the purity theory.
 In particular, an additive category $\sA$ \emph{need not} have weak
pushouts or weak pullbacks.
 For example, the existence of weak pullbacks in the category
$\sA=R\Modl_\inj$ of injective left modules over a ring $R$ is
equivalent to the existence of injective precovers of all left
$R$\+modules (in the sense of the paper~\cite{En}).
 By~\cite[Propositions~2.1 and~2.2]{En}, injective precovers exist
in $R\Modl$ if and only if $R$ is left Noetherian.
 Dually, the existence of weak pushouts in the category
$\sA=R\Modl_\proj$ of projective left $R$\+modules is equivalent to
the existence of projective preenvelopes of all left $R$\+modules.
 By the argument of~\cite[proofs of Propositions~2.1 and~5.1]{En},
the latter condition implies that the infinite direct products of
projective left $R$\+modules are projective, which does not hold for
most rings~$R$ (cf.~\cite[Theorem~P]{Bas}, \cite[Theorem~3.3]{Cha}).

 The aim of this paper is to spell out reasonable conditions on
a $\kappa$\+accessible category $\sA$ that (1)~hold for all
$\kappa$\+accessible additive categories, and (2)~imply good properties
of $\kappa$\+pure monomorphisms and $\kappa$\+pure epimorphisms.
 The reader will see that the resulting conditions are indeed quite
mild.

 Let us emphasize that \emph{some} assumptions are certainly necessary
for the purity theory in nonadditive categories.
 In particular, \cite[Example~2.5]{AR3} provides an example of
a finitely accessible category with a pure monomorphism that is
\emph{not} a directed colimit of split monomorphisms and \emph{not}
a regular monomorphism.
 In Examples~\ref{no-pushouts-of-split-monos-in-acc-preadd-example}
and~\ref{no-pullbacks-of-split-epis-in-acc-preadd-example}, we present
an essentially trivial example of an accessible preadditive (but not
additive!) category in which all monomorphisms and epimorphisms are
split, but some pushouts of monomorphisms and some pullbacks of
epimorphisms do not exist.

 In the context of a $\kappa$\+accessible category $\sA$, we use
the terminology \emph{strongly $\kappa$\+pure monomorphisms} for
the morphisms in $\sA$ that can be obtained as $\kappa$\+directed
colimits of split monomorphisms of $\kappa$\+presentable objects
in~$\sA$.
 Similarly, the \emph{strongly $\kappa$\+pure epimorphisms} are
the $\kappa$\+directed colimits of split epimorphisms between
$\kappa$\+presentable objects.
 We start with establishing very mild sufficient conditions for all
$\kappa$\+pure monomorphisms and $\kappa$\+pure epimorphisms to be
strongly $\kappa$\+pure.
 Then we proceed to provide further, also mild sufficient conditions for
strongly $\kappa$\+pure mono/epimorphisms to be regular and preserved
by pushouts/pullbacks.

 As a generalization of $\kappa$\+pure monomorphisms and $\kappa$\+pure
epimorphisms, we discuss what we call \emph{QE\+mono} and \emph{QE\+epi}
classes of morphisms (where QE means ``Quillen exact'').
 These are nonadditive generalizations of \emph{right exact} and
\emph{left exact} categories inroduced by Rump~\cite[Definition~4
in Section~5]{Rum} and studied by Bazzoni and Crivei~\cite{BC}.
 In the terminology of Henrard and van~Roosmalen~\cite{HR},
the latter (additive categories with additional structure) are called
\emph{inflation-exact} and \emph{deflation-exact} categories.
 See Rosenberg's preprints~\cite[Section~1.1]{Ros1},
\cite[Chapter~I]{Ros2} for prior art in the context of nonadditive
categories.

 Given a $\kappa$\+accessible category $\sA$ with the full subcategory
of $\kappa$\+presentable objects $\sA_{<\kappa}\subset\sA$, and given
a QE\+mono class $\cM$ or QE\+epi class $\cP$ in $\sA_{<\kappa}$, we
prove that the class of all $\kappa$\+directed colimits of morphisms
from $\cM$ (respectively, from $\cP$) is a QE\+mono (resp., QE\+epi)
class of morphisms in~$\sA$.
 This provides a nonadditive generalization of the results
of~\cite[Sections~1\+-2]{Plce}, and simultaneously their extension from
the setting of exact categories (in the sense of Quillen~\cite{Bueh})
to that of right exact and left exact (additive) categories.

 Notice that the classical notion of an exact category in the sense of
Quillen does not seem to make much sense in the nonadditive setting.
 The point is that exact categories are additive categories with
a class of \emph{admissible short exact sequences} $0\rarrow K\overset f
\rarrow L\overset g\rarrow M\rarrow0$, which, first of all, have to be
\emph{kernel-cokernel pairs}: $f=\ker g$ and $g=\coker f$.
 In the context of nonadditive categories, one does not usually
consider kernel-cokernel pairs of morphisms.

 In the general (nonadditive) category theory, there is a natural 
construction of the (\emph{co})\emph{equalizer} of a parallel pair of
morphisms $\bullet\rightrightarrows\bullet$, which is a single morphism
$\bullet\rarrow\bullet$.
 Conversely, to a single morphism $\bullet\rarrow\bullet$, one
assigns it (\emph{co})\emph{kernel pair}, which is a parallel pair
of morphisms $\bullet\rightrightarrows\bullet$.
 So, instead of a single self-dual concept of a kernel-cokernel pair,
in the nonadditive realm there are two concepts, dual to each other,
represented by diagrams of the shape
$$
 \bullet\rarrow\bullet\rightrightarrows\bullet
 \quad\text{or}\quad
 \bullet\rightrightarrows\bullet\rarrow\bullet.
$$
 Accordingly, it seems to be natural to split the single concept
of an exact category in the sense of Quillen into two halves
(the forementioned right exact and left exact categories) before
extending in to the nonadditive world.
 This is the approach that we follow in the present paper.

 Let us mention that, in spite of our discussion above, a self-dual
nonadditive version of Quillen exact categories exists in
the literature, introduced by Dyckerhoff and Kapranov under
the name of \emph{proto-exact categories}~\cite[Section~2.4]{DK}.
 We do not consider this concept in the present paper.

\subsection*{Acknowledgement}
 I~am grateful to Jan \v St\!'ov\'\i\v cek, Jan Trlifaj, and
Ji\v r\'\i\ Rosick\'y for organizing the joint Brno--Prague workshop
in Brno in June~2025 and inviting me to give a talk there, which was
a major stimulating experience for the present research.
 I~also wish to thank an anonymous referee for several helpful comments.
 In particular, Lemmas~\ref{effective-monomorphisms-lemma}
and~\ref{effective-epimorphisms-lemma} were suggested to me by
the referee.
 The author is supported by the GA\v CR project 23-05148S
and the Institute of Mathematics, Czech Academy of Sciences
(research plan RVO:~67985840).

\Section{Preliminaries on Accessible Categories}
\label{preliminaries-secn}

 We use the book~\cite{AR} as the background reference source on
accessible categories.
 In particular, we refer to~\cite[Definition~1.4, Theorem
and Corollary~1.5, Definition~1.13(1), and Remark~1.21]{AR} for
a discussion of \emph{$\lambda$\+directed} vs.\
\emph{$\lambda$\+filtered} colimits.
 For an earlier exposition avoiding a small mistake in~\cite[proof of
Theorem~1.5]{AR}, see~\cite{AN}.

 Let $\kappa$~be a regular cardinal and $\sA$ be a category with
$\kappa$\+directed (equivalently, $\kappa$\+filtered) colimits.
 An object $S\in\sA$ is called
\emph{$\kappa$\+presentable}~\cite[Definition~1.13(2)]{AR}
if the covariant functor $\Hom_\sA(S,{-})\:\sA\rarrow\Sets$ from $\sA$
to the category of sets $\Sets$ preserves $\kappa$\+directed colimits.
 We denote the full subcategory of $\kappa$\+presentable objects
by $\sA_{<\kappa}\subset\sA$.

 The category $\sA$ is called
\emph{$\kappa$\+accessible}~\cite[Definition~2.1]{AR} if there is
a \emph{set} of $\kappa$\+presentable objects $\sS\subset\sA$
such that all the objects of $\sA$ are $\kappa$\+directed colimits
of objects from~$\sS$.
 If this is the case, then the $\kappa$\+presentable objects of $\sA$
are precisely all the retracts of the objects from~$\sS$.

 A category is called \emph{accessible} if it is
$\kappa$\+accessible for some regular cardinal~$\kappa$.
 In the case of the countable cardinal $\kappa=\aleph_0$, one speaks of
\emph{finitely accessible categories}~\cite[Remark~2.2(1)]{AR}.

 Given a class of objects $\sT\subset\sA$, we denote by
$\varinjlim_{(\kappa)}\sT\subset\sA$ the class (or the full subcategory)
of all objects of $\sA$ that can be obtained as $\kappa$\+directed
colimits of objects from~$\sT$.
 The following proposition is well-known.

\begin{prop} \label{accessible-subcategory}
 Let\/ $\sA$ be a $\kappa$\+accessible category and\/ $\sT\subset\sA$
be a set of (some) $\kappa$\+presentable objects.
 Then the full subcategory\/ $\sB=\varinjlim_{(\kappa)}\sT\subset\sA$
is closed under $\kappa$\+directed colimits in~$\sA$.
 The category\/ $\sB$ is $\kappa$\+accessible, and
the $\kappa$\+presentable objects of\/ $\sB$ are precisely all
the retracts of the objects from\/ $\sT$.
 Equivalently, the $\kappa$\+presentable objects of\/ $\sB$ are
precisely all the objects of\/ $\sB$ that are $\kappa$\+presentable
in\/~$\sA$.
 An object $A\in\sA$ belongs to\/ $\sB$ if and only if, for every
object $S\in\sA_{<\kappa}$, every morphism $S\rarrow A$ in\/ $\sA$
factorizes through an object from\/~$\sT$.
\end{prop}

\begin{proof}
 In the context of finitely accessible additive categories, this result
goes back to~\cite[Proposition~2.1]{Len}, \cite[Section~4.1]{CB},
and~\cite[Proposition~5.11]{Kra}.
 For the full generality, see, e.~g., \cite[Proposition~1.2]{Pacc}.
\end{proof}

 Let $\sA$, $\sB$, and $\sC$ be three categories, and let
$F\:\sA\rarrow\sC$ and $G\:\sB\rarrow\sC$ be two functors.
 Following~\cite[Notation~2.42]{AR}, we denote by $F\down G$
the category of all triples $(A,B,h)$, where $A\in\sA$ and $B\in\sB$
are two objects and $h\:F(A)\rarrow G(B)$ is a morphism in~$\sC$.
 Morphisms in the category $F\down G$ are defined in the obvious way.

\begin{prop} \label{comma-category-accessible}
 Let $\sA$, $\sB$, and $\sC$ be $\kappa$\+accessible categories, and
let $F\:\sA\rarrow\sC$ and $G\:\sB\rarrow\sC$ be functors preserving
$\kappa$\+directed colimits and taking $\kappa$\+presentable objects
to $\kappa$\+presentable objects.
 Then the category $F\down G$ is $\kappa$\+accessible.
 An object $(S,T,u)\in F\down G$ is $\kappa$\+presentable if and only
if the object $S$ is $\kappa$\+presentable in $\sA$ and the object $T$
is $\kappa$\+presentable in\/~$\sB$.
\end{prop}

\begin{proof}
 This is~\cite[proof of Theorem~2.43]{AR}; see
also~\cite[Proposition~A.3]{Plce}.
\end{proof}

 The following proposition is a slightly stronger version of
Proposition~\ref{comma-category-accessible}.

\begin{prop} \label{comma-category-objects-directed-colimits}
 Let $\sA$, $\sB$, and $\sC$ be $\kappa$\+accessible categories, and
let $F\:\sA\rarrow\sC$ and $G\:\sB\rarrow\sC$ be functors preserving
$\kappa$\+directed colimits and taking $\kappa$\+presentable objects
to $\kappa$\+presentable objects.
 Let\/ $\sS\subset\sA_{<\kappa}$ and\/ $\sT\subset\sB_{<\kappa}$ be
some chosen subsets of $\kappa$\+presentable objects in\/ $\sA$ and\/
$\sB$ such that all objects of\/ $\sA$ are $\kappa$\+directed colimits
of objects from\/ $\sS$ and all objects of\/ $\sB$ are
$\kappa$\+directed colimits of objects from\/~$\sT$.
 Then all objects of $F\down G$ are $\kappa$\+directed colimits of
objects $(S,T,u)\in F\down G$ with $S\in\sS$ and $T\in\sT$.
\end{prop}

\begin{proof}
 This is what is actually proved in~\cite[proof of Theorem~2.43]{AR}.
\end{proof}

 A category $D$ is said to be \emph{finite} if the set of all morphisms
in $D$ is finite.
 More generally, a category $D$ is said to be \emph{$\kappa$\+small}
if the cardinality of the set of all morphisms in $D$ is smaller
than~$\kappa$.

\begin{prop} \label{rigid-finite-diagram-category-accessible}
 Let $\sA$ be a $\kappa$\+accessible category, and let $D$ be a finite
category in which all endomorphisms of objects are identity morphisms.
 Then the category\/ $\sA^D$ of all (covariant) functors $D\rarrow\sA$
is $\kappa$\+accessible.
 A functor $F\:D\rarrow\sA$ is $\kappa$\+presentable as an object of\/
$\sA^D$ if and only if, for every object $d\in D$, the object $F(d)$ is
$\kappa$\+presentable in\/~$\sA$.
\end{prop}

\begin{proof}
 In the case of finitely accessible categories~$\sA$, this result goes
back to~\cite[Expos\'e~I, Proposition~8.8.5]{SGA4}
and~\cite[page~55]{Mey}.
 For an arbitrary regular cardinal~$\kappa$, the desired assertion is
a particular case of~\cite[Theorem~1.3]{Hen}.
 See also~\cite[Proposition~A.5]{Plce}.
\end{proof}

 We use the notation $\varprojlim$ and $\varinjlim$ for limits and
colimits in categories.
 The upper index, such as in $\varprojlim^\sA$ and $\varinjlim^\sA$,
is used to indicate that the (co)limit is taken in the category~$\sA$.
 By \emph{$\kappa$\+small} (\emph{co})\emph{limits} one means
(co)limits of diagrams indexed by $\kappa$\+small indexing
categories~$D$.

\begin{lem} \label{small-colimits-preserved}
 Let\/ $\sA$ be a $\kappa$\+accessible category.
 Then the full subcategory\/ $\sA_{<\kappa}\subset\sA$ of all
$\kappa$\+presentable objects in\/ $\sA$ is closed under all
$\kappa$\+small colimits that exist in\/~$\sA$.
 Furthermore, the fully faithful inclusion functor\/ $\sA_{<\kappa}
\rarrow\sA$ preserves all $\kappa$\+small colimits that exist
in\/~$\sA_{<\kappa}$.
\end{lem}

\begin{proof}
 This follows from the fact that $\kappa$\+directed colimits
commute with $\kappa$\+small limits in the category of sets.
 For the first assertion, see~\cite[Proposition~1.16]{AR}.
 To prove the second claim, let $D$ be a $\kappa$\+small
category and let $F\:D\rarrow\sA_{<\kappa}$ be a $D$\+indexed
diagram in $\sA_{<\kappa}$ with the colimit
$A=\varinjlim_{d\in D}^{\sA_{<\kappa}}F(d)\in\sA_{<\kappa}$
computed in the category~$\sA_{<\kappa}$.
 Let $B\in\sA$ be an arbitrary object, and let
$B=\varinjlim_{\xi\in\Xi}^\sA S_\xi$ be a representation of $B$ as
the colimit of a diagram of objects $S_\xi\in\sA_{<\kappa}$,
indexed by a $\kappa$\+directed poset $\Xi$, the colimit being
computed in the category~$\sA$.
 Then in the category of sets we have
\begin{multline*}
 \varprojlim\nolimits_{d\in D}^\Sets\Hom_\sA(F(d),B)=
 \varprojlim\nolimits_{d\in D}^\Sets
 \varinjlim\nolimits_{\xi\in\Xi}^\Sets
 \Hom_{\sA_{<\kappa}}(F(d),S_\xi) \\
 =\varinjlim\nolimits_{\xi\in\Xi}^\Sets
 \varprojlim\nolimits_{d\in D}^\Sets
 \Hom_{\sA_{<\kappa}}(F(d),S_\xi)
 =\varinjlim\nolimits_{\xi\in\Xi}^\Sets
 \Hom_{\sA_{<\kappa}}(A,S_\xi)=\Hom_\sA(A,B),
\end{multline*}
as desired.
\end{proof}

 In the terminology of~\cite[Example~6.38]{AR}, full subcategories
$\sS\subset\sA$ satisfying the assumptions of the next lemma are called
\emph{weakly colimit-dense}.

\begin{lem} \label{limits-preserved}
 Let\/ $\sA$ be a category and\/ $\sS\subset\sA$ be a full subcategory
such that the minimal full subcategory of\/ $\sA$ containing\/ $\sS$
and closed under those colimits that exist in\/ $\sA$, coincides
with\/~$\sA$.
 Then the fully faithful inclusion functor\/ $\sS\rarrow\sA$ preserves
all those limits that exist in\/~$\sS$.
\end{lem}

\begin{proof}
 This follows from the fact that limits commute with limits (in any
category, and in particular) in the category of sets.
 Let $D$ be a small category, let $G\:D\rarrow\sS$ be a diagram
indexed by $D$, and let $B=\varprojlim_{d\in D}^\sS G(d)\in\sS$ be
the limit of the diagram $G$ computed in the category~$\sS$.
 We have to prove that the natural map $\Hom_\sA(A,B)\rarrow
\varprojlim_{d\in D}^\Sets\Hom_\sA(A,G(d))$ is a bijection of sets
for all objects $A\in\sA$.

 Denote by $\sE$ the full subcategory of $\sA$ consisting of all
objects $E$ for which the map of sets $\Hom_\sA(E,B)\rarrow
\varprojlim_{d\in D}^\Sets\Hom_\sA(E,G(d))$ is bijective.
 By assumption, we know that $\sS\subset\sE$, and it remains to check
that the full subcategory $\sE\subset\sA$ is closed under those
colimits that exist in~$\sA$.

 Let $C$ be a small category, let $F\:C\rarrow\sE$ be a diagram indexed
by $C$, and let $A=\varinjlim_{c\in C}^\sA F(c)\in\sA$ be the colimit
of the diagram $F$ computed in the category~$\sA$.
 So the map $\Hom_\sA(F(c),B)\rarrow
\varprojlim_{d\in D}^\Sets\Hom_\sA(F(c),G(d))$
is a bijection of sets for all $c\in C$.
 Then it follows that the map $\Hom_\sA(A,B)\rarrow
\varprojlim_{d\in D}^\Sets\Hom_\sA(A,G(d))$ is a bijection as well.
 Indeed, we have
\begin{multline*}
 \varprojlim\nolimits_{d\in D}^\Sets\Hom_\sA(A,G(d))
 =\varprojlim\nolimits_{d\in D}^\Sets
 \varprojlim\nolimits_{c\in C}^\Sets
 \Hom_\sA(F(c),G(d)) \\
 = \varprojlim\nolimits_{c\in C}^\Sets
 \varprojlim\nolimits_{d\in D}^\Sets
 \Hom_\sA(F(c),G(d))=
 \varprojlim\nolimits_{c\in C}^\Sets\Hom_\sA(F(c),B)=\Hom_\sA(A,B).
\end{multline*}
\end{proof}

\begin{lem} \label{presentables-strong-generator}
 Let\/ $\sA$ be a $\kappa$\+accessible category and $A\rarrow B$ be
a morphism in\/ $\sA$ such that the induced map of sets\/ $\Hom_\sA(S,A)
\rarrow\Hom_\sA(S,B)$ is bijective for all $\kappa$\+presentable
objects $S\in\sA$.
 Then the morphism $A\rarrow B$ is an isomorphism in\/~$\sA$.
\end{lem}

\begin{proof}
 This follows from the fact that every object $C\in\sA$ is the colimit
of the canonical diagram of morphisms into $C$ from
$\kappa$\+presentable objects of~$\sA$ \,\cite[Section~0.6, Definition
and Remark~1.23, Remark~2.2(4), and Proposition~2.8(i)]{AR}.
 See also~\cite[Lemma~1.1]{Pacc}.
\end{proof}

\begin{lem} \label{directed-colimits-commute-with-small-limits}
 In any $\kappa$\+accessible category\/ $\sA$, $\kappa$\+directed
colimits commute with those $\kappa$\+small limits that exist
in\/~$\sA$.
 Specifically, if\/ $\Xi$ is $\kappa$\+directed poset, $D$ is
a $\kappa$\+small category, and $F\:\Xi\times D\rarrow\sA$ is
a functor such that the limit\/ $\varprojlim_{d\in D}^\sA F(\xi,d)$
exists in\/ $\sA$ for all $\xi\in\Xi$, then
\begin{equation} \label{directed-colimits-small-limits}
 \varprojlim\nolimits_{d\in D}^\sA
 \varinjlim\nolimits_{\xi\in\Xi}^\sA F(\xi,d)=
 \varinjlim\nolimits_{\xi\in\Xi}^\sA
 \varprojlim\nolimits_{d\in D}^\sA F(\xi,d).
\end{equation}
\end{lem}

\begin{proof}
 This is the generalization of~\cite[Proposition~1.59]{AR} from
locally presentable to accessible categories.
 As in Lemma~\ref{small-colimits-preserved}, the basic explanation for
why this assertion holds is because $\kappa$\+directed colimits commute
with $\kappa$\+small limits in the category of sets.
 If the limit in the left-hand side
of~\eqref{directed-colimits-small-limits} exists, then the assertion
that the natural morphism from the right-hand side to the left-hand
side is an isomorphism follows easily by applying
Lemma~\ref{presentables-strong-generator}.
 Notice that the limits in the right-hand side
of~\eqref{directed-colimits-small-limits} exist by the assumptions
of the present lemma.

 When one wants to prove the existence of the limit in the left-hand
side of~\eqref{directed-colimits-small-limits} rather than assume it,
the following argument works.
 For every object $A\in\sA$, we need to show that the natural map
of sets
$$
 f_A\:\Hom_\sA(A,\>\varinjlim\nolimits_{\xi\in\Xi}^\sA
 \varprojlim\nolimits_{d\in D}^\sA F(\xi,d))\lrarrow
 \varprojlim\nolimits_{d\in D}^\Sets
 \Hom_\sA(A,\>\varinjlim\nolimits_{\xi\in\Xi}^\sA F(\xi,d))
$$
is a bijection.
 When the object $A\in\sA$ is $\kappa$\+presentable, we use the facts
that the covariant functor $\Hom_\sA(A,{-})$ takes both limits and
$\kappa$\+directed colimits in $\sA$ to the respective (co)limits in
$\Sets$ in order to reduce the question to the previously mentioned
assertion that $\kappa$\+directed colimits commute with $\kappa$\+small
limits in the category of sets.
 In the general case, the object $A$ is a ($\kappa$\+directed) colimit
of $\kappa$\+presentable objects, and it remains to point out
that both the domain and the codomain of the map~$f_A$, viewed as
contravariant functors $\sA^\sop\rarrow\Sets$ of the varying object
$A\in\sA$, take colimits in $\sA$ to limits in $\Sets$.
\end{proof}

\Section{Very Weak Cokernel Pairs}
\label{very-weak-cokernel-pairs-secn}

 Let $\sC$ be a category.
 Given a pair of morphisms $i\:A\rarrow B$ and $p\:B\rarrow A$ in
$\sC$ such that the composition $p\circ i=\id_A$ is the identity
morphism, one says that $i$~is a \emph{split monomorphism} and
$p$~is a \emph{split epimorphism} in~$\sC$.

 By a \emph{pushout} in $\sC$ one means the colimit of a diagram of
the shape
\begin{equation} \label{span-diagram}
\begin{gathered}
 \xymatrix{
 B \\
 A \ar^g[r] \ar_f[u] & C
 }
\end{gathered}
\end{equation}
 A \emph{cokernel pair} is a pushout of the diagram as above with
$B=C$ and $f=g$.
 So the cokernel pair of a morphism $f\:A\rarrow B$ in $\sC$ is
a parallel pair of morphisms $k_1$, $k_2\:B\rightrightarrows K$ such 
that $k_1\circ f=k_2\circ f$ and the triple $(K,k_1,k_2)$ is universal
with this property in the category~$\sC$.

 The definition of a \emph{weak colimit} is obtained from the usual
definition of a colimit by dropping the condition of uniqueness of
the required morphism and keeping only the existence.
 Specifically, let $D$ be a small category and $F\:D\rarrow\sC$ be
a $D$\+indexed diagram in~$\sC$.
 Let $A\in\sC$ be an object and $F\rarrow A$ be a compatible cocone
(i.~e., in other words, a morphism from $F$ to the constant
$D$\+indexed diagram in $\sC$ corresponding to the object~$A$).
 Then one says that $A$ is a weak colimit of $F$ if, for every
object $B\in\sC$ and any compatible cocone $F\rarrow B$, there exists
a (not necessarily unique) morphism $A\rarrow B$ in $\sC$ making
the triangular diagram $F(d)\rarrow A\rarrow B$ commutative in $\sC$
for all $d\in D$.

 As particular cases of the general definition of a weak colimit,
one can speak about weak pushouts, weak cokernel pairs, etc.

 Let $f\:A\rarrow B$ be a morphism and $k_1$, $k_2\:
B\rightrightarrows K$ be a weak cokernel pair of~$f$.
 The parallel pair of identity morphisms $\id_B$, $\id_B\:B
\rightrightarrows B$ obviously has the property that the two morphisms
have equal compositions with the morphism~$f$.
 Consequently, there exists a morphism $s\:K\rarrow B$ such that
$s\circ k_1=\id_B=s\circ k_2$.
 Thus both the morphisms $k_1$ and $k_2\:B\rarrow K$ are split
monomorphisms.

\begin{defn} \label{very-weak-cokernel-pair-definition}
 Let $f\:A\rarrow B$ be a morphism in $\sC$ and $c_1$, $c_2\:
B\rightrightarrows C$ be a parallel pair of morphisms such that
$c_1\circ f=c_2\circ f$.
 We will say that a parallel pair of morphisms $k_1$, $k_2\:B
\rightrightarrows K$ in $\sC$ is a \emph{very weak cokernel pair}
of~$f$ \emph{with respect to $(c_1,c_2)$} if the following three
conditions hold:
\begin{itemize}
\item one has $k_1\circ f=k_2\circ f$;
\item there exists a morphism $l\:K\rarrow C$ such that
$c_1=l\circ k_1$ and $c_2=l\circ k_2$;
\item the morphism $k_1\:B\rarrow K$ is a split monomorphism
(i.~e., there exists a morphism $s\:K\rarrow B$ such that
$s\circ k_1=\id_B$).
\end{itemize}
\end{defn}

 The commutative diagram described in
Definition~\ref{very-weak-cokernel-pair-definition} can be drawn as
\begin{equation} \label{very-weak-cokernel-pair-diagram}
\begin{gathered}
 \xymatrix{
  B \ar[rr]^{c_1} \ar@{>..>}[rd]^{k_1}&& C \\
  & K \ar@{..>}[ru]_l \\
  A \ar[uu]^f \ar[rr]_f &&
  B\ar[uu]_{c_2} \ar@{..>}[lu]_{k_2}
 }
\end{gathered}
\end{equation}
 Here the splitting~$s$ of the split monomorphism~$k_1$ is not depicted
on the diagram~\eqref{very-weak-cokernel-pair-diagram};
instead, the condition that $k_1$~is a split monomorphism is
expressed by the tail at the beginning of the dotted arrow
showing~$k_1$.
 
\begin{exs} \label{very-weak-cokernel-pair-examples}
 (1)~If the morphism $f\:A\rarrow B$ has a weak cokernel pair
$k_1$, $k_2\:B\rightrightarrows K$, then $(k_1,k_2)$ is a very weak
cokernel pair of~$f$ with respect to every parallel pair of
morphisms $(c_1,c_2)$ such that $c_1\circ f=c_2\circ f$.
 This is clear from the discussion above.
 In this sense, our terminology is consistent.

\smallskip
 (2)~Let $f\:A\rarrow B$ and $c_1$, $c_2\:B\rightrightarrows C$ be
three morphisms such that $c_1\circ f=c_2\circ f$.
 Assume that the product $K=B\times C$ exists in $\sC$, and denote
by $p_B\:K\rarrow B$ and $p_C\:K\rarrow C$ the product projections.
 Let $k_i\:B\rarrow K$, \,$i=1$, $2$, be the morphisms for which
$p_B\circ k_i=\id_B$ and $p_C\circ k_i=c_i$.
 Then $(k_1,k_2)$ is a very weak cokernel pair of~$f$ with respect
to~$(c_1,c_2)$.
 Indeed, the equation $c_1\circ f=c_2\circ f$ implies
$k_1\circ f=k_2\circ f$ by the uniqueness condition in the universal
property of the product.
 In the notation of
Definition~\ref{very-weak-cokernel-pair-definition}, it remains
to put $l=p_C$ and $s=p_B$.
\end{exs}

 We will say that a category $\sC$ \emph{has very weak cokernel pairs}
if for any three morphisms $f\:A\rarrow B$ and $c_1$, $c_2\:
B\rightrightarrows C$ such that $c_1\circ f=c_2\circ f$ in $\sC$ there
exists a very weak cokernel pair of~$f$ with respect to~$(c_1,c_2)$
in~$\sC$.

\begin{rem} \label{flat-modules-counterex-remark}
 By Example~\ref{very-weak-cokernel-pair-examples}(2), any category
with finite products has very weak cokernel pairs.
 In particular, any additive category has very weak cokernel pairs.

 Notice, however, that an accessible additive category \emph{need not} 
have weak cokernel pairs in general.
 For example, let $R$ be an associative ring, and consider the additive
category of flat left $R$\+modules $\sA=R\Modl_\flat$.
 It is well known that the category $\sA$ is finitely accessible.

 Given an arbitrary left $R$\+module $M$, pick a morphism of flat left
$R$\+modules $f\:A\rarrow B$ such that $M$ is the cokernel of~$f$
in the abelian category $R\Modl$.
 Denote by $m\:B\rarrow M$ the natural epimorphism in $R\Modl$.
 Let $k_1$, $k_2\:B\rightrightarrows K$ be a weak cokernel pair of~$f$
in~$\sA$.
 Then we have $(k_2-k_1)\circ f=0$, hence there exists a morphism
$e\:M\rarrow K$ in $R\Modl$ such that $k_2-k_1=e\circ m$.
 We claim that the morphism~$e$ is a flat preenvelope of $M$, in
the sense of~\cite{En}.

 Indeed, let $g\:M\rarrow L$ be a morphism from $M$ to a flat
left $R$\+module $L$.
 Consider the pair of morphisms $l_1=0\:B\rarrow L$ and
$l_2=g\circ m\:B\rarrow L$.
 Then we have $l_1\circ f=0=g\circ m\circ f=l_2\circ f$.
 By assumption, there exists a morphism $h\:K\rarrow L$ such that
$l_1=h\circ k_1$ and $l_2=h\circ k_2$.
 Hence $g\circ m=l_2-l_1=h\circ (k_2-k_1)=h\circ e\circ m$.
 As the morphism~$m$ is an epimorphism in $R\Modl$, it follows
that $g=h\circ e$.
 Thus the morphism~$g$ factorizes through~$e$, as desired.

 Conversely, if flat preenvelopes exist in $R\Modl$, then all weak
colimits exist in $\sA=R\Modl_\flat$.
 Indeed, given a diagram $F\:D\rarrow\sA$, denote by $M$ the colimit
of $F$ is $R\Modl$.
 Then any flat preenvelope of $M$ is a weak colimit of $F$ in~$\sA$.

 We have shown that weak cokernel pairs exist in $\sA$ if and only if
flat preenvelopes exist in $R\Modl$.
 The latter property holds if and only if the ring $R$ is right
coherent~\cite[Proposition~5.1]{En}.
 Taking a ring $R$ that is \emph{not} right coherent, we obtain
an example of a finitely accessible additive category $\sA$ without
weak cokernel pairs.
\end{rem}

\begin{ex} \label{no-very-weak-cokernel-pair-example}
 Here is an example of a preadditive but not additive category
(i.~e., a category enriched in abelian groups but not having finite
products or finite coproducts) which does \emph{not} even have very
weak cokernel pairs.
 Let $\k$ be a field, $n\ge1$ be an integer, $\k\Vect$ be the category
of $\k$\+vector spaces, and $\sA\subset\k\Vect $ be the full subcategory
of $\k$\+vector spaces of finite dimension not exceeding~$n$.
 For any nonnegative integer~$i$, let $\k^i$ denote the $\k$\+vector
space of dimension~$i$.
 Let $A=0$, $B=\k^n$, and $C=\k^i\in\sA$, where $0<i\le n$.
 Let $f\:A\rarrow B$ be the zero morphism and $c_1$, $c_2\:
B\rightrightarrows C$ be a parallel pair of morphisms such that
$c_1=0$ and $c_2\ne0$.
 Then, of course, $c_1\circ f=c_2\circ f$.
 However, the morphism~$f$ does \emph{not} have a very weak cokernel
pair with respect to~$(c_1,c_2)$.
 Indeed, assume for the sake of contradiction that $k_1$, $k_2\:
B\rightrightarrows K$ is such a very weak cokernel pair.
 Let $l\:K\rarrow C$ and $s\:K\rarrow B$ be the related morphisms.
 So $k_1$~is a split monomorphism, $s\circ k_1=\id_B$.
 Since $B=\k^n$ and the category $\sA$ contains no vector spaces of
dimension greater than~$n$, the morphism~$k_1$ has to be an isomorphism.
 Then the equation $0=c_1=l\circ k_1$ implies $l=0$, which makes
the equation $0\ne c_2=l\circ k_2$ impossible to satisfy.
\end{ex}

\begin{lem} \label{very-weak-cokernel-pairs-in-kappa-presentables}
 Let $\kappa$~be a regular cardinal and\/ $\sA$ be
a $\kappa$\+accessible category with very weak cokernel pairs.
 Then the full subcategory\/ $\sA_{<\kappa}\subset\sA$ of
$\kappa$\+presentable objects in\/ $\sA$ also has very weak cokernel
pairs.
\end{lem}

\begin{proof}
 Let $f\:A\rarrow B$ and $c_1$, $c_2\:B\rightrightarrows C$ be three
morphisms in $\sA_{<\kappa}$ such that $c_1\circ\nobreak f=c_2\circ f$.
 Let $x_1$, $x_2\:B\rightrightarrows X$ be a very weak cokernel pair
of~$f$ with respect to~$(c_1,c_2)$ in the category~$\sA$, and let
$y\:X\rarrow C$ and $t\:X\rarrow B$ be the related morphisms.
 Let $X=\varinjlim_{\xi\in\Xi}K_\xi$ be a representation of $X$ as
a $\kappa$\+directed colimit of $\kappa$\+presentable objects $K_\xi$
in~$\sA$, indexed by a $\kappa$\+directed poset~$\Xi$.
 Denote by $w_{\eta\xi}\:K_\xi\rarrow K_\eta$ (where $\xi\le\eta$
in $\Xi$) the transition morphisms and by $z_\xi\:K_\xi\rarrow X$
the canonical morphisms into the colimit.

 Since $\Xi$ is $\kappa$\+directed and $B$ is $\kappa$\+presentable,
there exists an index $\xi\in\Xi$ such that both the morphisms
$x_1$ and $x_2\:B\rightrightarrows X$ factorize through
the morphism~$z_\xi$.
 So we have a parallel pair of morphisms $k'_1$, $k'_2\:
B\rightrightarrows K_\xi$ such that $x_i=z_\xi\circ k'_i$
for $i=1$,~$2$.
 By assumption, we have $x_1\circ f=x_2\circ f$, so $z_\xi\circ k'_1
\circ f=z_\xi\circ k'_2\circ f$.
 Since $\Xi$ is $\kappa$\+directed and $A$ is $\kappa$\+presentable,
there exists an index $\eta\in\Xi$, \,$\xi\le\eta$, such that
$w_{\eta\xi}\circ k'_1\circ f=w_{\eta\xi}\circ k'_2\circ f$.
 Put $k_i=w_{\eta\xi}\circ k'_i$ for $i=1$, $2$, and $K=K_\eta$.

 Then $k_1$, $k_2\:B\rightrightarrows K$ is a very weak cokernel pair
of~$f$ with respect to $(c_1,c_2)$ in the category~$\sA_{<\kappa}$.
 Indeed, we have already seen that $k_1\circ f=k_2\circ f$.
 In the notation of
Definition~\ref{very-weak-cokernel-pair-definition}, it remains to
put $l=y\circ z_\eta$ and $s=t\circ z_\eta$.
\end{proof}

\Section{Strongly Pure Monomorphisms}
\label{strongly-pure-monomorphisms-secn}

 Let $\kappa$~be a regular cardinal and $\sA$ be a $\kappa$\+accessible
category.
 A morphism $m\:C\rarrow D$ is said to be a \emph{$\kappa$\+pure
monomorphism}~\cite[Definition~2.27]{AR} in $\sA$ if, for every morphism
$S\rarrow T$ in $\sA_{<\kappa}$ and any commutative square diagram
\begin{equation} \label{pure-monomorphism-definition-diagram}
\begin{gathered}
 \xymatrix{
  C \ar[r]^m & D \\
  S \ar[u]^c \ar[r]_t & T \ar[u]_d
 }
 \qquad\qquad
 \xymatrix{
  C \\
  S \ar[u]^c \ar[r]_t & T \ar@{..>}[lu]_e
 }
\end{gathered}
\end{equation}
in~$\sA$, there exists a morphism $e\:T\rarrow C$ making the lower
triangle commutative.

 Let $I=(\bullet\to\bullet)$ be the category with two objects and
one nonidentity morphism (acting from one object of $I$ to
the other one).
 Given a category $\sC$, we denote by $\sC^\to=\sC^I$ the category
of functors $I\rarrow\sC$, i.~e., the category of morphisms in~$\sC$.
 So the objects of $\sC^\to$ are all the morphisms in $\sC$, and
the morphisms in $\sC^\to$ are all the commutative squares in~$\sC$.

\begin{lem} \label{pure-monomorphisms-simple-properties}
\textup{(a)} All split monomorphisms in\/ $\sA$ are $\kappa$\+pure
monomorphisms. \par
\textup{(b)} All $\kappa$\+pure monomorphisms are monomorphisms
in~$\sA$. \par
\textup{(c)} The class of $\kappa$\+pure monomorphisms is closed under
compositions of morphisms in\/~$\sA$. \par
\textup{(d)} If $i$, $j$ is a composable pair of morphisms in\/ $\sA$
and $i\circ j$ is a $\kappa$\+pure monomorphism, then $j$~is
a $\kappa$\+pure monomorphism. \par
\textup{(e)} The class of $\kappa$\+pure monomorphisms in\/ $\sA$ is
closed under $\kappa$\+directed colimits in\/~$\sA^\to$.
\end{lem}

\begin{proof}
 All the assertions are well-known.
 Parts~(a) and~(c\+-d), mentioned in~\cite[Example~2.28(1) and
Remarks~2.28(1\+-2)]{AR}, are elementary.
 Part~(b) is~\cite[Proposition~2.29]{AR}.
 In part~(e), which is~\cite[Proposition~2.30(i)]{AR}, one needs to
use the fact that all the objects of $(\sA_{<\kappa})^\to$ are
$\kappa$\+presentable in~$\sA^\to$.
\end{proof}

 It follows from
Lemma~\ref{pure-monomorphisms-simple-properties}(a,e) that
$\kappa$\+directed colimits of split monomorphisms are
$\kappa$\+pure monomorphisms in~$\sA$.
 The aim of this section is to provide a mild sufficient condition
for the inverse implication.
 In fact, we will prove a little bit more.

 Let us say that a morphism $m\:C\rarrow D$ in $\sA$ is
a \emph{strongly $\kappa$\+pure monomorphism} if $m$~is
a $\kappa$\+directed colimit in $\sA^\to$ of split monomorphisms
between $\kappa$\+presentable objects in~$\sA$.

\begin{prop} \label{strongly-pure-monomorphisms-characterized}
 A morphism~$m$ in\/ $\sA$ is a strongly $\kappa$\+pure monomorphism
if and only if any morphism into~$m$ from a morphism in\/
$\sA_{<\kappa}$ factorizes through a split monomorphism in\/
$\sA_{<\kappa}$ in the category\/ $\sA^\to$.
 In other words, $m$~is a strongly $\kappa$\+pure monomorphism if and
only if, for any morphism~$t$ in\/ $\sA_{<\kappa}$ and any morphism
$t\rarrow m$ in $\sA^\to$ there exists a split monomorphism~$s$
in\/ $\sA_{<\kappa}$ such that the morphism $t\rarrow m$ factorizes
as $t\rarrow s\rarrow m$ for some morphisms $t\rarrow s$
and $s\rarrow m$ in\/~$\sA^\to$.
\end{prop}

\begin{proof}
 According to
Proposition~\ref{rigid-finite-diagram-category-accessible} for $D=I$
(or to Proposition~\ref{comma-category-accessible} for $\sA=\sB=\sC$),
the category $\sA^\to$ is $\kappa$\+accessible, and
the $\kappa$\+presentable objects of $\sA^\to$ are precisely all
the morphisms between $\kappa$\+presentable objects in $\sA$,
that is $(\sA^\to)_{<\kappa}=(\sA_{<\kappa})^\to$.
 (This observation can be found in~\cite[Exercise~2.c]{AR}.)
 The desired assertion is now provided by
Proposition~\ref{accessible-subcategory} applied to
the $\kappa$\+accessible category $\sA^\to$ and the set of
$\kappa$\+presentable objects $\sT$ consisting of all
the (representatives of isomorphism classes) of split monomorphisms
in~$\sA_{<\kappa}$.
\end{proof}

\begin{lem} \label{strongly-pure-monos-directed-colimit-closed}
 The class of strongly $\kappa$\+pure monomorphisms in\/ $\sA$ is
closed under $\kappa$\+directed colimits in\/ $\sA^\to$.
\end{lem}

\begin{proof}
 This is another assertion from
Proposition~\ref{accessible-subcategory}, applicable to the situation
at hand as explained in the proof of
Proposition~\ref{strongly-pure-monomorphisms-characterized}.
\end{proof}

\begin{thm} \label{under-very-weak-cokernel-pairs-theorem}
 In any $\kappa$\+accessible category\/ $\sA$ with very weak cokernel
pairs, the classes of $\kappa$\+pure monomorphisms and strongly
$\kappa$\+pure monomorphisms coincide.
 In other words, all $\kappa$\+pure monomorphisms in\/ $\sA$ are
$\kappa$\+directed colimits of split monomorphisms between
$\kappa$\+presentable objects of\/~$\sA$.
\end{thm}

\begin{proof}
 Let $m\:C\rarrow D$ be a $\kappa$\+pure monomorphism in $\sA$,
and let $t\:S\rarrow T$ be a morphism in $\sA_{<\kappa}$.
 Suppose we are given a morphism $t\rarrow m$ in~$\sA^\to$;
this means a commutative square diagram as
in~\eqref{pure-monomorphism-definition-diagram}.
 In view of Proposition~\ref{strongly-pure-monomorphisms-characterized},
we need to prove that the morphism $(c,d)\:t\rarrow m$ factorizes
through some split monomorphism $u\:U\rarrow V$ in $\sA_{<\kappa}$,
viewed as an object of $\sA^\to$.

 The following argument is a nonadditive version
of~\cite[proof of Lemma~4.3]{Plce}.
 By assumption, there exists a lifting $e\:T\rarrow C$ such that
$c=e\circ t$, as on the triangular diagram
in~\eqref{pure-monomorphism-definition-diagram}.
 Consider the parallel pair of morphisms $m\circ e$,
$d\:T\rightrightarrows D$.
 We have $m\circ e\circ t=m\circ c=d\circ t$.
 
 Let $D=\varinjlim_{\xi\in\Xi}W_\xi$ be a representation of $D$
as a $\kappa$\+directed colimit of $\kappa$\+presentable objects
$W_\xi$, indexed by a $\kappa$\+directed poset~$\Xi$.
 Denote by $w_{\eta\xi}\:W_\xi\rarrow W_\eta$ the transition
morphisms (for $\xi$, $\eta\in\Xi$, \,$\xi\le\eta$) and
by $z_\xi\:W_\xi\rarrow D$ the canonical morphisms to the colimit.
 Since $T\in\sA_{<\kappa}$, there exists an index $\xi\in\Xi$
such that both the morphisms $m\circ e$ and $d\:T\rightrightarrows D$
factorize through the morphism $z_\xi\:W_\xi\rarrow D$.
 So we have a parallel pair of morphisms $b'_1$, $b'_2\:
T\rightrightarrows W_\xi$ such that $m\circ e=z_\xi\circ b'_1$
and $d=z_\xi\circ b'_2$.
 Now $z_\xi\circ b'_1\circ t=m\circ e\circ t=d\circ t=z_\xi\circ
b'_2\circ t$.
 Since $S\in\sA_{<\kappa}$, there exists an index $\eta\in\Xi$,
\,$\xi\le\eta$, such that $w_{\eta\xi}\circ b'_1\circ t=
w_{\eta\xi}\circ b'_2\circ t$.
 Put $b_i=w_{\eta\xi}\circ b'_i$, \,$i=1$,~$2$, and $W=W_\eta$.
 We have constructed a parallel pair of morphisms $b_1$, $b_2\:
T\rightrightarrows W$ such that $m\circ e=z_\eta\circ b_1$ and
$d=z_\eta\circ b_2$, where $W\in\sA_{<\kappa}$ and
$z_\eta\:W\rarrow D$.
 Furthermore, we have $b_1\circ t=b_2\circ t$.
 So we obtain a commutative diagram
$$
 \xymatrix{
  C \ar[r]^m & D \\
  T \ar[u]^e \ar[r]^{b_1} & W \ar[u]_{z_\eta} \\
  S \ar[u]^t \ar[r]^t \ar@/^2pc/[uu]^c
  & T\ar[u]_{b_2}\ar@/_2.5pc/[uu]_d
 }
$$

 Finally, by assumption, very weak cokernel pairs exist in $\sA$,
and by Lemma~\ref{very-weak-cokernel-pairs-in-kappa-presentables} it
follows that very weak cokernel pairs exist in $\sA_{<\kappa}$ as well.
 Let $k_1$, $k_2\:T\rightrightarrows V$ be a very weak cokernel pair of
the morphism $t\:S\rarrow T$ with respect to the parallel pair
of morphisms $b_1$, $b_2\:T\rightrightarrows W$ in
the category~$\sA_{<\kappa}$.
 Let $l\:V\rarrow W$ and $s\:V\rarrow T$ be the related morphisms,
as per Definition~\ref{very-weak-cokernel-pair-definition}.
 Put $u=k_1$ and $v=k_2$, and also $z=z_\eta l$ and $U=T$.
 So, in particular, $u$~is a split monomorphism, $s\circ u=\id_T$.
 We have arrived to a commutative diagram
$$
 \xymatrix{
  C \ar[r]^m & D \\
  T \ar[u]^e \ar@{>->}[r]^u & V \ar[u]_z \\
  S \ar[u]^t \ar[r]^t \ar@/^2pc/[uu]^c
  & T\ar[u]_v \ar@/_2pc/[uu]_d
 }
$$
providing the desired factorization $t\rarrow u\rarrow m$,
in the category $\sA^\to$, of the morphism $(c,d)\:t\rarrow m$
through a split monomorphism of $\kappa$\+presentable objects
$u\:U=T\rarrow V$ in the category~$\sA$.
\end{proof}

\begin{cor} \label{under-finite-products-corollary}
 In any $\kappa$\+accessible category\/ $\sA$ with finite products,
the classes of $\kappa$\+pure monomorphisms and strongly
$\kappa$\+pure monomorphisms coincide.
 In other words, all $\kappa$\+pure monomorphisms in\/ $\sA$ are
$\kappa$\+directed colimits of split monomorphisms between
$\kappa$\+presentable objects of\/~$\sA$.
\end{cor}

\begin{proof}
 This is a particular case of
Theorem~\ref{under-very-weak-cokernel-pairs-theorem},
in view of Example~\ref{very-weak-cokernel-pair-examples}(2).
\end{proof}

\Section{Very Weak Split Pullbacks}
\label{very-weak-split-pullbacks-secn}

 Let $\sC$ be a category.
 Dually to the discussion of weak colimits in
Section~\ref{very-weak-cokernel-pairs-secn}, one defines the notion
of a \emph{weak limit} of a diagram in~$\sC$.

 A diagram of the shape
\begin{equation} \label{cospan-diagram}
\begin{gathered}
 \xymatrix{
  B \ar[r]^f & A \\
  & C \ar[u]_g
 }
\end{gathered}
\end{equation}
is called a \emph{cospan}.
 The limits of cospans are called the \emph{pullbacks}, and accordingly
weak limits of cospans are called weak pullbacks.

 We will say that a cospan~\eqref{cospan-diagram} is \emph{split} if
a morphism $h\:C\rarrow B$ exists making the triangular diagram
\begin{equation} \label{split-cospan-diagram}
\begin{gathered}
 \xymatrix{
  B \ar[r]^f & A \\
  & C \ar[u]_g \ar@{..>}[lu]^h
 }
\end{gathered}
\end{equation}
commutative.
 We will say that \emph{weak split pullbacks exist} in a category $\sC$
if all split cospans have weak pullbacks.

 Let $P$ be a weak pullback of a split
cospan~\eqref{split-cospan-diagram}, and let $p_B\:P\rarrow B$ and
$p_C\:P\rarrow C$ be the canonical morphisms from the weak limit.
 Then it is clear from the definitions that there exists a morphism
$s\:C\rarrow P$ such that $p_B\circ s=h$ and $p_C\circ s=\id_C$.
 Therefore, the morphism~$p_C$ is a split epimorphism.

\begin{defn} \label{very-weak-split-pullback-definition}
 Suppose we are given a split cospan
diagram~\eqref{split-cospan-diagram}
in $\sC$ together with a commutative square as on the diagram
\begin{equation} \label{very-weak-split-pullback-setting-posing-diags}
\begin{gathered}
 \xymatrix{
  B \ar[r]^f & A \\
  & C \ar[u]_g \ar[lu]^h
 }
 \qquad\qquad
 \xymatrix{
  B \ar[r]^f & A \\
  Q \ar[u]^{q_B} \ar[r]_{q_C} & C \ar[u]_g
 }
\end{gathered}
\end{equation}
 We will say that an object $P\in\sC$ together with a pair of
morphisms $p_B\:P\rarrow B$ and $p_C\:P\rarrow C$ is a \emph{very
weak split pullback} of~$(f,g)$ \emph{with respect to~$(q_B,q_C)$}
if the following three conditions hold:
\begin{itemize}
\item one has $f\circ p_B=g\circ p_C$;
\item there exists a morphism $r\:Q\rarrow P$ such that
$q_B=p_B\circ r$ and $q_C=p_C\circ r$;
\item the morphism $p_C\:P\rarrow C$ is a split epimorphism (i.~e.,
there exists a morphism $s\:C\rarrow P$ such that $p_C\circ s=\id_C$).
\end{itemize}
\end{defn}

 The commutative diagram described in
Definition~\ref{very-weak-split-pullback-definition} can be drawn as
\begin{equation} \label{very-weak-split-pullback-diagram}
\begin{gathered}
 \xymatrix{
  B\ar[rr]^f && A \\
  & P \ar@{..>}[lu]_{p_B} \ar@{..>>}[rd]^{p_C} \\
  Q\ar[uu]^{q_B} \ar[rr]_{q_C} \ar@{..>}[ru]^r
  && C \ar[uu]_g
 }
\end{gathered}
\end{equation}
 Here the splitting~$s$ of the split epimorphism~$p_C$ is not depicted
on the diagram~\eqref{very-weak-split-pullback-diagram}; instead,
the condition that~$p_C$ is a split epimorphism is expressed by
the double head at the end of the dotted arrow showing~$p_C$.

\begin{rem}
 Dually to cospans, a diagram of the shape~\eqref{span-diagram} is
called a \emph{span}.
 Notice that any span consisting of two equal morphisms, $B=C$ and
$f=g$, is a split span (in the sense dual to our definition of
a split cospan in this section).
 For this reason, using the terminology \emph{very weak split pushouts}
for the notion dual to very weak split pullbacks, one can observe that
the very weak cokernel pairs from
Definition~\ref{very-weak-cokernel-pair-definition} are a special case
of very weak split pushouts.
\end{rem}

\begin{exs} \label{very-weak-split-pullback-examples}
 (1)~If a split cospan $f\:B\rarrow A$, \ $g\:C\rarrow A$ has a weak
pullback $p_B\:P\rarrow B$, \ $p_C\:P\rarrow C$, then $(p_B,p_C)$ is
a very weak split pullback of $(f,g)$ with respect to any pair of
morphisms $(q_B,q_C)$ such that $f\circ q_B=g\circ q_C$.
 This is clear from the discussion above.
 In this sense, our terminology is consistent.
 
\smallskip
 (2)~Suppose we are given a commutative triangle and a commutative
square as in~\eqref{very-weak-split-pullback-setting-posing-diags}.
 Assume that the coproduct $P=Q\sqcup C$ exists in~$\sC$,
and denote by $i_Q\:Q\rarrow P$ and $i_C\:C\rarrow P$ the coproduct
injections.
 Let $p_B\:P\rarrow B$ be the morphism for which $p_B\circ i_Q=q_B$
and $p_B\circ i_C=h$, and let $p_C\:P\rarrow C$ be the morphism
for which $p_C\circ i_Q=q_C$ and $p_C\circ i_C=\id_C$.
 Then $(p_B,p_C)$ is a very weak split pullback of $(f,g)$ with
respect to~$(q_B,q_C)$.
 Indeed, the equations $f\circ q_B=g\circ q_C$ and $f\circ h=g$
imply $f\circ p_B=g\circ p_C$ by the uniqueness condition in
the universality property of the coproduct.
 In the notation of
Definition~\ref{very-weak-split-pullback-definition},
it remains to put $r=i_Q$ and $s=i_C$.
\end{exs}

 We will say that a category $\sC$ \emph{has very weak split pullbacks}
if for every commutative triangle and commutative square as
in~\eqref{very-weak-split-pullback-setting-posing-diags} there exists
a very weak split pullback of $(f,g)$ with respect to~$(q_B,q_C)$
in~$\sC$.

\begin{rem}
 By Example~\ref{very-weak-split-pullback-examples}(2), any category
with finite coproducts has very weak split pullbacks.
 In particular, any additive category has very weak split pullbacks.

 Notice, however, that an accessible additive category \emph{need not}
have weak split pullbacks in general.
 For example, dually to Remark~\ref{flat-modules-counterex-remark},
let $R$ be an associative ring, and consider the additive category
of injective left $R$\+modules $\sA=R\Modl_\inj$.
 It is well known that the category $\sA$ is accessible.
 In fact, $\sA$ is $\kappa$\+accessible whenever $\lambda$~is a regular
cardinal such every left ideal in $R$ has less than~$\lambda$
generators, $\nu$~is any infinite cardinal greater than or equal to
the cardinality of $R$ such that the set $\nu^{<\lambda}$ of all subsets
of~$\nu$ of the cardinality smaller than~$\lambda$ has cardinality equal
to~$\nu$, that is $\nu^{<\lambda}=\nu$, and $\kappa=\nu^+$ is
the successor cardinal of~$\nu$ (see, e.~g.,
\cite[Corollary~3.7]{Pres}).

 Given an arbitrary left $R$\+module $M$, pick a morphism of injective
left $R$\+modules $f\:B\rarrow A$ such that $M$ is the kernel of~$f$
in the abelian category $R\Modl$.
 Denote by $m\:M\rarrow B$ the natural monomorphism in $R\Modl$.
 Put $C=0$, and let $g\:C\rarrow A$ be the zero morphism.
 Then $(f,g)$ is obviously a split cospan in~$\sA$.
 Let $p_B\:P\rarrow B$ and $p_C\:P\rarrow C$ be a weak pullback of
$(f,g)$ in~$\sA$.
 Then we have $f\circ p_B=0$, hence there exists a morphism
$c\:P\rarrow M$ in $R\Modl$ such that $p_B=m\circ c$.
 We claim that the morphism~$c$ is an injective precover of $M$,
in the sense of~\cite{En}.

 Indeed, let $k\:Q\rarrow M$ be a morphism into $M$ from an injective
left $R$\+module~$Q$.
 Consider the pair of morphisms $q_B=m\circ k\:Q\rarrow B$ and
$q_C=0\:Q\rarrow C$.
 Then we have $f\circ q_B=f\circ m\circ k=0=g\circ q_C$.
 By assumption, there exists a morphism $r\:Q\rarrow P$ such that
$q_B=p_B\circ r$.
 Hence $m\circ k=q_B=p_B\circ r=m\circ c\circ r$.
 As the morphism~$m$ is a monomorphism in $R\Modl$, it follows that
$k=c\circ r$.
 Thus the morphism~$k$ factorizes through~$c$, as desired.

 Conversely, if injective precovers exist in $R\Modl$, then all
weak limits exist in $\sA=R\Modl_\inj$.
 Indeed, given a diagram $F\:D\rarrow\sA$, denote by $M$ the limit
of $F$ in $R\Modl$.
 Then any injective precover of $M$ is a weak limit of $F$ in~$\sA$.
 
 We have shown that weak split pullbacks exist in $\sA$ if and only if
injective precovers exist in $R\Modl$.
 The latter property holds if and only if the ring $R$ is left
Noetherian~\cite[Propositions~2.1 and~2.2]{En}.
 Taking a ring $R$ that is \emph{not} left Noetherian, we obtain
an example of an accessible additive category $\sA$ without weak
split pullbacks.
\end{rem}

\begin{ex} \label{no-very-weak-split-pullback-pair-example}
 Dually to Example~\ref{no-very-weak-cokernel-pair-example},
the preadditive category $\sA$ of $\k$\+vector spaces of dimension
not exceeding~$n$ (where $n\ge1$ is a fixed integer) does \emph{not}
have very weak split pullbacks.
 Specifically, put $A=0$, $B=\k^j$ for some $0<j\le n$, \ $C=\k^n$,
and $Q=\k^i$ for some $0<i\le n$.
 Let $f\:B\rarrow A$ and $g\:C\rarrow A$ be the zero morphisms,
and let $q_B\:Q\rarrow B$ and $q_C\:Q\rarrow C$ be any morphisms
such that $q_C=0$ and $q_B\ne0$.
 Then $f\circ q_B=g\circ q_C$, but the split cospan $(f,g)$ has
no very weak split pullback with respect to $(q_B,q_C)$ in~$\sA$.
 Indeed, if $p_C\:P\rarrow C$ is a split epimorphism in $\sA$,
then $p_C$ is an isomorphism, so the equation $0=q_C=p_C\circ r$ implies
$r=0$, which is incompatible with $0\ne q_B=p_B\circ r$.
\end{ex}

\begin{lem} \label{very-weak-split-pullbacks-in-kappa-presentables}
 Let $\kappa$~be a regular cardinal and\/ $\sA$ be
a $\kappa$\+accessible category with very weak split pullbacks.
 Then the full subcategory\/ $\sA_{<\kappa}\subset\sA$ of
$\kappa$\+presentable objects in\/ $\sA$ also has very weak split
pullbacks.
\end{lem}

\begin{proof}
 Suppose we are given a commutative triangle and a commutative
square~\eqref{very-weak-split-pullback-setting-posing-diags}
in the category~$\sA_{<\kappa}$.
 Let $x_B\:X\rarrow B$ and $x_C\:X\rarrow C$ be a very weak split
pullback of $(f,g)$ with respect to $(q_B,q_C)$ in the category~$\sA$,
and let $y\:Q\rarrow X$ and $t\:C\rarrow X$ be the related morphisms.
 Let $X=\varinjlim_{\xi\in\Xi}P_\xi$ be a representation of $X$ as
a $\kappa$\+directed colimit of $\kappa$\+presentable objects $P_\xi$
in~$\sA$, indexed by a $\kappa$\+directed poset~$\Xi$.
 Denote by $z_\xi\:P_\xi\rarrow X$ the canonical morphisms into
the colimit.

 Since $\Xi$ is $\kappa$\+directed and $Q$ and $C$ are
$\kappa$\+presentable, there exists an index $\xi\in\Xi$ such that
both the morphisms $y$ and~$t$ factorize through the morphism~$z_\xi$.
 So we have morphisms $r\:Q\rarrow P_\xi$ and $s\:C\rarrow P_\xi$
such that $y=z_\xi\circ r$ and $t=z_\xi\circ s$.
 Put $p_B=x_B\circ z_\xi$ and $p_C=x_C\circ z_\xi$, and $P=P_\xi$.

 Then $p_B\:P\rarrow B$ and $p_C\:P\rarrow C$ is a very weak split
pullback of $(f,g)$ with respect to $(q_B,q_C)$ in
the category~$\sA_{<\kappa}$.
 Indeed, we have $f\circ p_B=f\circ x_B\circ z_\xi=g\circ x_C\circ z_\xi
=g\circ p_C$.
 Furthermore, $p_B\circ r=x_B\circ z_\xi\circ r=x_B\circ y=q_B$
and $p_C\circ r=x_C\circ z_\xi\circ r=x_C\circ y=q_C$.
 Finally, $p_C\circ s=x_C\circ z_\xi\circ s=x_C\circ t=\id_C$.
\end{proof}

\Section{Strongly Pure Epimorphisms}
\label{strongly-pure-epimorphisms-secn}

 Let $\kappa$~be a regular cardinal and $\sA$ be a $\kappa$\+accessible
category.
 A morphism $p\:D\rarrow E$ is said to be a \emph{$\kappa$\+pure
epimorphism}~\cite[Definition~1]{AR2} in $\sA$ if for every
$\kappa$\+presentable object $S$ and any morphism $e\:S\rarrow E$
in $\sA$ there exists a morphism $l\:S\rarrow D$ making the triangular
diagram
\begin{equation} \label{pure-epimorphism-definition-diagram}
\begin{gathered}
 \xymatrix{
  D \ar[r]^p & E \\
  & S \ar[u]_e \ar@{..>}[lu]^l
 }
\end{gathered}
\end{equation}
commutative.

 We refer to Section~\ref{strongly-pure-monomorphisms-secn} for
the notation $\sC^\to=\sC^I$ for a category~$\sC$.

\begin{lem} \label{pure-epimorphisms-simple-properties}
\textup{(a)} All split epimorphisms in\/ $\sA$ are $\kappa$\+pure
epimorphisms. \par
\textup{(b)} All $\kappa$\+pure epimorphisms are epimorphisms
in\/~$\sA$. \par
\textup{(c)} The class of $\kappa$\+pure epimorphisms is closed under
compositions of morphisms in\/~$\sA$. \par
\textup{(d)} If $p$, $q$ is a composable pair of morphisms in\/ $\sA$
and $p\circ q$ is a $\kappa$\+pure epimorphism, then $p$~is
a $\kappa$\+pure epimorphism. \par
\textup{(e)} The class of $\kappa$\+pure epimorphisms in\/ $\sA$
is closed under $\kappa$\+directed colimits in\/~$\sA^\to$.
\end{lem}

\begin{proof}
 All the assertions are well-known and easy to prove.
 Parts~(a) and~(e) are~\cite[Example~2(a\+-b)]{AR2}.
 Part~(b) is~\cite[Proposition~4(a)]{AR2}.
\end{proof}

 It follows from
Lemma~\ref{pure-epimorphisms-simple-properties}(a,e) that
$\kappa$\+directed colimits of split epimorphisms are
$\kappa$\+pure epimorphisms in~$\sA$.
 The aim of this section is to provide a mild sufficient condition
for the inverse implication.
 In fact, we will prove a little bit more.

 Let us say that a morphism $p\:D\rarrow E$ in $\sA$ is
a \emph{strongly $\kappa$\+pure epimorphism} if $p$~is
a $\kappa$\+directed colimit in $\sA^\to$ of split epimorphisms
between $\kappa$\+presentable objects in~$\sA$.

\begin{prop} \label{strongly-pure-epimorphisms-characterized}
 A morphism~$p$ in\/ $\sA$ is a strongly $\kappa$\+pure epimorphism
if and only if any morphism into~$p$ from a morphism in\/
$\sA_{<\kappa}$ factorizes through a split epimorphism in\/
$\sA_{<\kappa}$ in the category\/ $\sA^\to$.
 In other words, $p$~is a strongly $\kappa$\+pure epimorphism if and
only if, for any morphism~$t$ in\/ $\sA_{<\kappa}$ and any morphism
$t\rarrow p$ in $\sA^\to$ there exists a split epimorphism~$s$
in\/ $\sA_{<\kappa}$ such that the morphism $t\rarrow m$ factorizes
as $t\rarrow s\rarrow p$ for some morphisms $t\rarrow s$
and $s\rarrow p$ in\/~$\sA^\to$.
\end{prop}

\begin{proof}
 Similar to the proof of
Proposition~\ref{strongly-pure-monomorphisms-characterized}.
\end{proof}

\begin{lem} \label{strongly-pure-epis-directed-colimit-closed}
 The class of strongly $\kappa$\+pure epimorphisms in\/ $\sA$ is
closed under $\kappa$\+directed colimits in\/ $\sA^\to$.
\end{lem}

\begin{proof}
 Similar to the proof of
Lemma~\ref{strongly-pure-monos-directed-colimit-closed}.
\end{proof}

\begin{thm} \label{under-very-weak-split-pullbacks-theorem}
 In any $\kappa$\+accessible category\/ $\sA$ with very weak split
pullbacks, the classes of $\kappa$\+pure epimorphisms and strongly
$\kappa$\+pure epimorphisms coincide.
 In other words, all $\kappa$\+pure epimorphisms in\/ $\sA$ are
$\kappa$\+directed colimits of split epimorphisms between
$\kappa$\+presentable objects of\/~$\sA$.
\end{thm}

\begin{proof}
 Let $p\:D\rarrow E$ be a $\kappa$\+pure epimorphism in $\sA$, and
let $t\:T\rarrow S$ be a morphism in~$\sA_{<\kappa}$.
 Suppose we are given a morphism $t\rarrow p$ in $\sA^\to$;
this means a commutative square diagram
\begin{equation} \label{into-pure-epimorphism-to-be-factorized-diagram}
\begin{gathered}
 \xymatrix{
  D \ar[r]^p & E \\
  T \ar[u]^d \ar[r]_t & S \ar[u]_e
 }
\end{gathered}
\end{equation}
 In view of Proposition~\ref{strongly-pure-epimorphisms-characterized},
we need to prove that the morphism $(d,e)\:t\rarrow p$ factorizes
through some split epimorphism $u\:U\rarrow V$ in $\sA_{<\kappa}$,
viewed as an object of~$\sA^\to$.

 The following argument is a nonadditive version
of~\cite[proofs of Lemmas~1.5, 2.3, and~4.1, and Proposition~4.2]{Plce}.
 By assumption, there exists a lifting $l\:S\rarrow D$ such that
$e=p\circ l$, as on diagram~\eqref{pure-epimorphism-definition-diagram}.
 In other words, this means that the pair of morphisms $(p,e)$ is
a split cospan in~$\sA$.

 Following the proof of
Proposition~\ref{strongly-pure-monomorphisms-characterized}
or~\cite[Exercise~2.c]{AR}, any morphism in $\sA$, viewed as an object
of $\sA^\to$, is a $\kappa$\+directed colimit of morphisms between
$\kappa$\+presentable objects.
 Let $p=\varinjlim_{\xi\in\Xi}w_\xi$ be a representation of the morphism
$p\:D\rarrow E$ as a $\kappa$\+directed colimit of morphisms
$w_\xi\:X_\xi\rarrow Y_\xi$, with $\kappa$\+presentable objects $X_\xi$
and $Y_\xi$, indexed by a $\kappa$\+directed poset~$\Xi$.
 Denote by $x'_{\eta\xi}\:X_\xi\rarrow X_\eta$ and $y'_{\eta\xi}\:
Y_\xi\rarrow Y_\eta$ the components of the transition morphisms
$(x'_{\eta\xi},y'_{\eta\xi})$, for all $\xi$, $\eta\in\Xi$,
\,$\xi\le\eta$.
 Denote also by $x_\xi\:X_\xi\rarrow D$ and $y_\xi\:Y_\xi\rarrow E$
the components of the canonical morphisms to the colimit
$(x_\xi,y_\xi)\:w_\xi\rarrow p$.

 Since $\Xi$ is $\kappa$\+directed and $T$ and $S$ are
$\kappa$\+presentable, there exists an index $\xi\in\Xi$ such that
both the morphisms $d$ and~$l$ factorize through~$x_\xi$, while
the morphism~$e$ factorizes through~$y_\xi$.
 So we have morphisms $t_X\:T\rarrow X_\xi$, \ $h\:S\rarrow X_\xi$,
and $g'\:S\rarrow Y_\xi$ such that $d=x_\xi\circ t_X$, \
$l=x_\xi\circ h$, and $e=y_\xi\circ g'$.
 Hence $y_\xi\circ w_\xi\circ t_X=p\circ x_\xi\circ t_X=p\circ d=
e\circ t=y_\xi\circ g'\circ t$ and $y_\xi\circ w_\xi\circ h=
p\circ x_\xi\circ h=p\circ l=e=y_\xi\circ g'$.
 Since $\Xi$ is $\kappa$\+directed and $T$ and $S$ are
$\kappa$\+presentable, there exists an index $\eta\in\Xi$,
\,$\xi\le\eta$, such that $y'_{\eta\xi}\circ w_\xi\circ t_X=
y'_{\eta\xi}\circ g'\circ t$ and $y'_{\eta\xi}\circ w_\xi\circ h
=y'_{\eta\xi}\circ g'$.

 Put $X=X_\xi$, \ $Y=Y_\eta$, \ $f=y'_{\eta\xi}\circ w_\xi\:X\rarrow Y$,
and $g=y'_{\eta\xi}\circ g'\:S\rarrow Y$.
 Then we have commutative diagrams in~$\sA$
$$
 \xymatrix{
  D \ar[r]^p & E \\
  X \ar[u]^{x_\xi} \ar[r]^f & Y \ar[u]_{y_\eta} \\
  T \ar[u]^{t_X} \ar[r]^t \ar@/^2.5pc/[uu]^d
  & S\ar[u]_g\ar@/_2.5pc/[uu]_e
 }
 \qquad\qquad
 \xymatrix{
  \ \\
  X \ar[r]^f & Y \\
  & S \ar[u]_g \ar[lu]^h
 }
$$
with objects $S$, $T$, $X$, $Y\in\sA_{<\kappa}$.
 Now the pair of morphisms $(f,g)$ is a split cospan in~$\sA_{<\kappa}$.
 For convenience, put $t_S=t$. 

 Finally, by assumption, very weak split pullbacks exist in $\sA$, and
by Lemma~\ref{very-weak-split-pullbacks-in-kappa-presentables} it
follows that very weak split pullbacks exist in $\sA_{<\kappa}$ as well.
 Let $u_X\:U\rarrow X$ and $u_S\:U\rarrow S$ be a very weak split
pullback of $(f,g)$ with respect to~$(t_X,t_S)$ in
the category~$\sA_{<\kappa}$.
 Let $r\:T\rarrow U$ and $s\:S\rarrow U$ be the related morphisms,
as per Definition~\ref{very-weak-split-pullback-definition}.
 So, in particular, $u_S$~is a split epimorphism, $u_S\circ s=\id_S$.
 We have arrived to a commutative diagram
$$
 \xymatrix{
  D \ar[r]^p & E \\
  X \ar[u]^{x_\xi} \ar[r]^f & Y \ar[u]_{y_\eta} \\
  U \ar[u]^{u_X} \ar@{->>}[r]^{u_S} & S \ar[u]_g \\
  T \ar[u]^r \ar[r]^{t=t_S} \ar@/^6pc/[uuu]^d \ar@/^3pc/[uu]^{t_X}
  & S\ar@{=}[u]\ar@/_3.5pc/[uuu]_e
 }
$$
providing the desired factorization $t\rarrow u\rarrow p$, in
the category $\sA^\to$, of the morphism $(d,e)\:t\rarrow p$ through
a split epimorphism of $\kappa$\+presentable objects $u=u_S\:U
\rarrow V=S$ in the category~$\sA$.
\end{proof}

\begin{cor} \label{under-finite-coproducts-corollary}
 In any $\kappa$\+accessible category\/ $\sA$ with finite coproducts,
the classes of $\kappa$\+pure epimorphisms and strongly
$\kappa$\+pure epimorphisms coincide.
 In other words, all $\kappa$\+pure epimorphisms in\/ $\sA$ are
$\kappa$\+directed colimits of split epimorphisms between
$\kappa$\+presentable objects of\/~$\sA$.
\end{cor}

\begin{proof}
 This is a particular case of
Theorem~\ref{under-very-weak-split-pullbacks-theorem},
in view of Example~\ref{very-weak-split-pullback-examples}(2).
\end{proof}

\begin{rem}
 The additional assumptions in
Theorems~\ref{under-very-weak-cokernel-pairs-theorem}
and~\ref{under-very-weak-split-pullbacks-theorem} (on top of
the assumption that $\sA$ is a $\kappa$\+accessible category) may be
mild, but they \emph{cannot} be completely dropped.
 Indeed, the counterexample in~\cite[Example~2.5]{AR3} shows that
a pure (i.~e., $\aleph_0$\+pure) monomorphism in a finitely
accessible (i.~e., $\aleph_0$\+accessible) category $\sA$ need not
be a colimit of split monomorphisms in general.

 A similar construction provides an example showing that
a pure epimorphism in a finitely accessible category $\sA$ need not
be a directed colimit of split epimorphisms, generally speaking.
 In order to obtain the desired counterexample, it suffices to
modify~\cite[Example~2.5]{AR3} as follows.
 In the notation of~\cite[Example~2.5]{AR}, when freely adding
morphisms $e_i\:B_i\rarrow A_{i+1}$, instead of imposing
the relations $e_im_i=a_{i,i+1}$, impose the relations
$m_{i+1}e_i=b_{i,i+1}$ for all $i=0$, $1$, $2$,~\dots{}
 Then the morphism $m=\varinjlim_i m_i\:\varinjlim_i A_i
\rarrow\varinjlim_i B_i$ becomes a pure epimorphism that is not
a colimit of split epimorphisms (in fact, all split epimorphisms are
isomorphisms in the resulting finitely accessible category~$\sA$).
\end{rem}

\Section{QE-Mono Classes}  \label{QE-mono-classes-secn}

 Let $\sC$ be a category.
 A morphism $m\:C\rarrow D$ in $\sC$ is called a \emph{regular
monomorphism} if $m$~is the equalizer of a parallel pair of
morphisms $e_1$, $e_2\:D\rightrightarrows E$.

 Clearly, every regular monomorphism is a monomorphism.
 Every split monomorphism is regular: if $s\:D\rarrow C$ is
a morphism such that $s\circ m=\id_C$, then $m$~is the equalizer
of the pair of morphisms $m\circ s$ and $\id_D\:D\rightrightarrows D$.

 A morphism $m\:C\rarrow D$ is said to be an \emph{effective
monomorphism} if $m$~has a cokernel pair
$k_1$, $k_2\:D\rightrightarrows K$ in $\sC$ and $m$~is
the equalizer of $(k_1,k_2)$.
 One can easily see that if a regular monomorphism~$m$ has
a cokernel pair $(k_1,k_2)$, then $m$~is the equalizer of $(k_1,k_2)$.
 So a monomorphism is effective if and only if it is regular
\emph{and} has a cokernel pair.

\begin{lem} \label{effective-monomorphisms-lemma}
 For any $\kappa$\+accessible category\/ $\sA$, the class of
effective monomorphisms is closed under $\kappa$\+directed colimits
in\/~$\sA^\to$.
\end{lem}

\begin{proof}
 In any category, colimits commute with colimits; in particular,
the $\kappa$\+directed colimits preserve cokernel pairs.
 So, if a morphism~$m$ is a $\kappa$\+directed colimit of
effective monomorphisms~$m_\xi$, then the colimit of the cokernel
pairs of~$m_\xi$ is the cokernel pair of~$m$.
 Here we are assuming that the index~$\xi$ ranges over
a $\kappa$\+directed poset~$\Xi$.
 By Lemma~\ref{directed-colimits-commute-with-small-limits}, in
a $\kappa$\+accessible category, the $\kappa$\+directed colimits
commute with $\kappa$\+small limits; in particular,
the $\kappa$\+directed colimits preserve equalizers.
 Thus $m$~is the equalizer of its cokernel pair.
\end{proof}

 Let $\cM$ be a class of morphisms in a category~$\sC$.
 We will say that $\cM$ is a \emph{QE\+mono class} in $\sC$ if
the following conditions are satisfied:
\begin{enumerate}
\renewcommand{\theenumi}{\roman{enumi}}
\item All pushouts of all morphisms from $\cM$ exists in $\sC$,
and the class $\cM$ is stable under pushouts.
 In other words, for any span (i.~e., a pair of morphisms with common
domain) $m\:C\rarrow D$, \ $f\:C\rarrow C'$ such that $m\in\cM$,
the pushout $D'$ exists, and the morphism $m'\:C'\rarrow D'$
belongs to~$\cM$,
\begin{equation} \label{QE-mono-defn-pushout-diagram}
\begin{gathered}
 \xymatrix{
  C' \ar@{..>}[r]^{m'} & D' \\
  C \ar[u]^f \ar[r]_m & D \ar@{..>}[u]_{f'}
 }
\end{gathered}
\end{equation}
\item In particular, condition~(i) implies that all morphisms from
$\cM$ have cokernel pairs in~$\sC$.
 It is further required that every morphism from $\cM$ is the equalizer
of its cokernel pair.
 In other words, all the morphisms from $\cM$ must be effective
monomorphisms.
\item All the identity morphisms in $\sC$ belong to $\cM$, and the class
of morphisms $\cM$ is closed under compositions.
\end{enumerate}

 It is clear from the preceding discussion that, assuming condition~(i),
condition~(ii) is equivalent to the condition that all morphisms
from $\cM$ are regular monomorphisms in~$\sC$.

\begin{exs} \label{QE-mono-classes-examples}
 (1)~If $\sC$ is an additive category, then a QE\+mono class of
morphisms in $\sC$ is the same thing as a structure of \emph{right
exact category} on $\sC$ in the sense of~\cite[Definition~3.1]{BC}.
 In the terminology of~\cite[Definition~2.2]{HR}, such additive
categories with an additional structure are called
\emph{inflation-exact categories}.

\smallskip
 (2)~If the category $\sC^\sop$ opposite to $\sC$ is regular
(in the sense of~\cite{BGO,Gra}), then the class of all regular
monomorphisms in $\sC$ is a QE\+mono class.
 See Example~\ref{QE-epi-classes-examples}(2) below.

\smallskip
 (3)~The compositions of split monomorphisms are always split
monomorphisms, and all split monomorphisms are regular
(as explained above).
 Furthermore, if the morphism~$m$ on the pushout
diagram~\eqref{QE-mono-defn-pushout-diagram} is a split monomorphism
with a splitting $s\:D\rarrow C$, \ $s\circ m=\id_C$, then, by
the definition of a pushout, there exists a unique morphism 
$s'\:D'\rarrow C'$ such that $s'\circ m'=\id_{C'}$ and
$s'\circ f'=f\circ s$ (because $\id_{C'}\circ f=f=f\circ s\circ m$;
cf.\ the discussion of split pullbacks in
Section~\ref{very-weak-split-pullbacks-secn}).
 So all pushouts of split monomorphisms are split monomorphisms.

 Therefore, the class $\cM$ of all split monomorphisms in
a category $\sC$ is a QE\+mono class if and only all pushouts of
split monomorphisms exist in~$\sC$.
\end{exs}

\begin{rem} \label{pushouts-of-split-monomorphisms-remark}
 One can easily see that all pushouts of split monomorphisms exist
in an additive category if and only if the category is \emph{weakly 
idempotent-complete} (in the sense of~\cite[Section~7]{Bueh}).
 However, the following simple example shows that a \emph{preadditive}
category that is idempotent-complete (in the sense
of~\cite[Section~6]{Bueh}; or which is the same, has split idempotents
in the sense of~\cite[Observation~2.4]{AR}) still need not have
pushouts of split monomorphisms.

 Let $\k$ be a field, $n\ge1$ be an integer, $\k\Vect$ be the category
of $\k$\+vector spaces, and $\sA\subset\k\Vect $ be the full subcategory
of $\k$\+vector spaces of finite dimension not exceeding~$n$.
 For any nonnegative integer~$i$, let $\k^i$ denote the $\k$\+vector
space of dimension~$i$.
 Then the pair of split monomorphisms (actually, direct summand
injections) $b\:A=\k^{n-1}\rarrow\k^n=B$ and $c\:A=\k^{n-1}\rarrow
\k^n=C$ does \emph{not} have a pushout in~$\sA$.
 The pushout of $b$ and~$c$ in $\k\Vect$ is isomorphic to $\k^{n+1}$,
which does not belong to~$\sA$.
 Morever, the pair of morphisms $b$ and~$c$ does not have a weak
pushout in~$\sA$; so weak cokernel pairs do not exist in~$\sA$.
 In fact, the morphism $b=c$ does not even have a very weak cokernel
pair in~$\sA$; cf.\ Example~\ref{no-very-weak-cokernel-pair-example}.
\end{rem}

 Let $\cM$ be a QE\+mono class in a category~$\sC$.
 For every morphism $m\:C\rarrow D$ belonging to $\cM$, consider its
cokernel pair $k_1$, $k_2\:D\rightrightarrows K$.
 By an \emph{$\cM$\+sequence} we mean a diagram
$$
 \xymatrix{
  C\ar[r]^m & D \ar@<2pt>[r]^{k_1}\ar@<-2pt>[r]_{k_2} & K
 }
$$
arising from some morphism $m\in\cM$ in this way.

 Denote by $J$ the category with three objects $1$, $2$, and $3$,
and four nonidentity morphisms $1\rarrow 2$, \ $2\rightrightarrows 3$,
and $1\rarrow 3$.
 So both the compositions $1\rarrow 2\rightrightarrows 3$ are equal
to one and the same morphism $1\rarrow 3$.
 Given a category $\sC$, we are interested in the category of
diagrams~$\sC^J$.
 In particular, for any QE\+mono class $\cM$ in $\sC$,
the $\cM$\+sequences form a subclass of objects of~$\sC^J$.

 Let $\sA$ be a $\kappa$\+accessible category.
 We will say that a QE\+mono class $\cM$ in $\sA$ is \emph{locally
$\kappa$\+coherent} if the $\cM$\+sequences in $\sA$ are precisely all
the $\kappa$\+directed colimits of $\cM$\+sequences in $\sA$ with
all the three terms $C$, $D$, $K$ belonging to~$\sA_{<\kappa}$.
 Here the $\kappa$\+directed colimits are taken in the diagram
category~$\sA^J$.
 The terminology ``locally $\kappa$\+coherent'' comes from
the paper~\cite[Section~1]{Plce}.

\begin{lem} \label{loc-coh-QE-mono-classes-directed-colimit-closed}
 For any locally $\kappa$\+coherent QE\+mono class $\cM$ in
a $\kappa$\+accessible category\/ $\sA$, the class of all
$\cM$\+sequences is closed under $\kappa$\+directed colimits
in\/~$\sA^J$.
\end{lem}

\begin{proof}
 By Proposition~\ref{rigid-finite-diagram-category-accessible},
the category $\sA^J$ is $\kappa$\+accessible, and its
$\kappa$\+presentable objects are precisely all the $J$\+shaped
diagrams in~$\sA_{<\kappa}$.
 Now the desired assertion follows from
Proposition~\ref{accessible-subcategory} applied to
the $\kappa$\+accessible category~$\sA^J$.
\end{proof}

\begin{lem} \label{loc-coh-QE-mono-classes-trivial-characterization}
 A QE\+mono class $\cM$ in a $\kappa$\+accessible category\/ $\sA$ is
locally $\kappa$\+coherent if and only if $\cM$ is precisely the class
of all $\kappa$\+directed colimits of the morphisms from $\cM$ whose
domains and codomains are $\kappa$\+presentable.
 Here the $\kappa$\+directed colimits are taken in the category of
morphisms\/~$\sA^\to$.
\end{lem}

\begin{proof}
 The point is that the full subcategory of $\kappa$\+presentable
objects\/ $\sA_{<\kappa}$ is closed under finite colimits in~$\sA$
(meaning those finite colimits that exist in\/~$\sA$); see
Lemma~\ref{small-colimits-preserved}.
 In particularly, $\sA_{<\kappa}$ is closed in $\sA$ under
the cokernel pairs of those morphisms that have cokernel pairs in~$\sA$.
 So, in the notation above, if $C$, $D\in\sA_{<\kappa}$, then
$K\in\sA_{<\kappa}$.
 Furthermore, all existing colimits commute with all existing colimits
in any category; in particular, $\kappa$\+directed colimits preserve
cokernel pairs in~$\sA$.
\end{proof}

\begin{prop} \label{restriction-of-loc-coh-to-presentables-is-QE-mono}
 Let $\cM$ be a locally $\kappa$\+coherent QE\+mono class in
a $\kappa$\+accessible category\/~$\sA$.
 Then the class $\cM\cap\sA_{<\kappa}^\to$ of all morphisms from
$\cM$ with $\kappa$\+presentable domains and codomains is a QE\+mono
class in the category\/~$\sA_{<\kappa}$.
 Furthermore, the class $\cM\cap\sA_{<\kappa}^\to$ is closed under
retracts in the category\/ $\sA_{<\kappa}^\to=(\sA^\to)_{<\kappa}
=(\sA_{<\kappa})^\to$.
\end{prop}

\begin{proof}
 Condition~(iii) obviously holds for $\cM\cap\sA_{<\kappa}^\to$ whenever
it holds for~$\cM$.
 Condition~(i) holds for $\cM\cap\sA_{<\kappa}^\to$ whenever it holds
for $\cM$, because the full subcategory $\sA_{<\kappa}$ is closed in
$\sA$ under all the pushouts that exist in~$\sA$
(by Lemma~\ref{small-colimits-preserved}).
 Finally, condition~(ii) holds for $\cM\cap\sA_{<\kappa}^\to$ whenever
conditions (i) and~(ii) hold for $\cM$, because any diagram in
$\sA_{<\kappa}$ that is an equalizer diagram in $\sA$ is also
an equalizer diagram in~$\sA_{<\kappa}$.
 These arguments do not even use the assumption of local
$\kappa$\+coherence of the QE\+mono class $\cM$ in~$\sA$.
 The class $\cM\cap\sA_{<\kappa}^\to$ is closed under retracts in
$\sA_{<\kappa}^\to$ because retracts are special cases of
$\kappa$\+directed colimits in\/ $\sA^\to$,
see~\cite[Observation~2.4]{AR}.
\end{proof}

\Section{Construction of Locally Coherent QE-Mono Classes}
\label{QE-mono-construction-secn}

 Let $D$ be the finite category
$$
 \xymatrix{
  3 \\
  1 \ar[u] \ar[r] & 2
 }
$$
i.~e., the category with three objects $1$, $2$, $3$ and two
nonidentity morphisms $1\rarrow 2$ and $1\rarrow 3$.
 By the \emph{category of spans} in a category $\sC$ we mean
the category of $D$\+shaped diagrams in $\sC$, that is,
the category of functors~$\sC^D$.

\begin{thm} \label{direct-limit-closure-of-QE-mono-theorem}
 Let\/ $\sA$ be a $\kappa$\+accessible category and $\cN$ be
a QE\+mono class in the category\/~$\sA_{<\kappa}$.
 Then the class\/ $\varinjlim_{(\kappa)}\cN\subset\sA^\to$ of all
$\kappa$\+directed colimits of morphisms from $\cN$ (the colimits
being taken in\/~$\sA^\to$) is a locally $\kappa$\+coherent
QE\+mono class in the category\/~$\sA$.
\end{thm}

\begin{proof}
 Put $\cM=\varinjlim_{(\kappa)}\cN$.
 In order to check that condition~(i) for the class $\cN$ in
$\sA_{<\kappa}$ implies condition~(i) for the class $\cM$ in~$\sA$,
let us show that all spans $(m,f)$ in $\sA$ (where $m\:C\rarrow D$
and $f\:C\rarrow C'$) such that $m\in\cM$, are $\kappa$\+directed
colimits, in the category of spans in~$\sA$, of spans $(n,g)$
in $\sA_{<\kappa}$ such that $n\in\cN$.

 Once again, we use the fact that the category of morphisms $\sA^\to$
is $\kappa$\+accessible, and its $\kappa$\+presentable objects are
precisely all the morphisms with $\kappa$\+presentable domains and
codomains.
 Let $\sM$ be the full subcategory in $\sA^\to$ whose objects are all
the morphisms belonging to~$\cM$.
 By Proposition~\ref{accessible-subcategory} applied to the category
$\sA^\to$, the category $\sM$ is $\kappa$\+accessible, and its
$\kappa$\+presentable objects are precisely all the retracts of
the morphisms belonging to~$\cN$.
 Now let $F\:\sM\rarrow\sA$ be the functor taking every morphism
$m\:C\rarrow D$ to its domain $C$, and let $G\:\sA\rarrow\sA$ be
the identity functor.
 Then the category $F\down G$ defined in
Section~\ref{preliminaries-secn} is precisely the category of all
spans $(m,f)$ in $\sA$ with $m\in\cM$.
 By Proposition~\ref{comma-category-accessible}, it follows that
the category $F\down G$ is $\kappa$\+accessible, and its
$\kappa$\+presentable objects are precisely all the spans $(n',g')$
in $\sA_{<\kappa}$ such that $n'$~is a retract of a morphism
belonging to~$\cN$.

 Moreover, let $\sS\subset\sM_{<\kappa}$ be a set of representatives
of the isomorphism classes of morphisms belonging to $\cN$, and let
$\sT\subset\sA_{<\kappa}$ be a set of representatives of the isomorphism
classes of $\kappa$\+presentable objects of~$\sA$.
 Then Proposition~\ref{comma-category-objects-directed-colimits}
applied to the sets of $\kappa$\+presentable objects $\sS$ and $\sT$
tells us that all spans $(m,f)\in\sA^D$ with $m\in\cM$ are
$\kappa$\+directed colimits, in the category of spans~$\sA^D$,
of spans $(n,g)\in(\sA_{<\kappa})^D$ with $n\in\cN$ (as desired).

 By the second assertion of Lemma~\ref{small-colimits-preserved},
any finite colimit in $\sA_{<\kappa}$ is also a colimit in~$\sA$.
 In particular, this applies to pushouts.
 So pushouts of the spans $(n,g)$ in $\sA_{<\kappa}$ remain pushouts
in~$\sA$.
 Since $\kappa$\+directed colimits always preserve pushouts, we
have shown that condition~(i) for $\cN$ in $\sA_{<\kappa}$ implies
condition~(i) for $\cM$ in~$\sA$.

 The proof of the assertion that conditions~(i) and~(ii) for $\cN$
in $\sA_{<\kappa}$ imply condition~(ii) for $\cM$ in $\sA$ is
somewhat similar, and based on the arguments above together with
Lemmas~\ref{limits-preserved} and~\ref{effective-monomorphisms-lemma}.
 By assumption, every morphism $n\in\cN$ is the equalizer in
$\sA_{<\kappa}$ of its cokernel pair in~$\sA_{<\kappa}$.
 We have already seen that the cokernel pair of~$n$ in~$\sA_{<\kappa}$
is also the cokernel pair of~$n$ in~$\sA$.
 By Lemma~\ref{limits-preserved}, any limit that exists in
$\sA_{<\kappa}$ is also a limit in~$\sA$.
 Therefore, $n$~is the equalizer in $\sA$ of its cokernel pair in~$\sA$.
 So $n$~is an effective monomorphism in~$\sA$.
 Now any morphism $m\in\cM$ is a $\kappa$\+directed colimit in
$\sA^\to$ of morphisms $n\in\cN$, and $\kappa$\+directed colimits of
effective monomorphisms in $\sA$ are effective monomorphisms in $\sA$
by Lemma~\ref{effective-monomorphisms-lemma}.

 Let us prove that conditions~(i) and~(iii) for $\cN$ in $\sA_{<\kappa}$
imply condition~(iii) for $\cM$ in~$\sA$.
 Obviously, every identity morphism in $\sA$ is a $\kappa$\+directed
colimit of identity morphisms in~$\sA_{<\kappa}$.
 It remains to show that the class $\cM$ is closed under compositions.
 Let $m'\:C\rarrow D$ and $m''\:D\rarrow E$ be two morphisms belonging
to~$\cM$.
 By Proposition~\ref{accessible-subcategory}, in order to prove that
the composition $m=m''\circ m'$ belongs to $\cM$, we need to check that
every morphism $t\rarrow m$ in $\sA^\to$ from a morphism~$t$ with
$\kappa$\+presentable domain and codomain, $t\in(\sA^\to)_{<\kappa}
=(\sA_{<\kappa})^\to$, into the morphism~$m$ factorizes as
$t\rarrow n\rarrow m$, where $n\in\cN$.
 So we have a commutative diagram
$$
 \xymatrix{
  C \ar[r]^{m'} & D\ar[r]^{m''} & E \\
  S \ar[u]^c \ar[rr]_t && T \ar[u]_e
 }
$$
in $\sA$ with $\kappa$\+presentable objects $S$ and~$T$.

 Considering the composition $m'\circ c\:S\rarrow D$ of two morphisms
$c\:S\rarrow C$ and $m'\:C\rarrow D$, have a morphism $(m'\circ c,\>e)\:
t\rarrow m''$ in the category~$\sA^\to$.
 Since the morphism $m''\:D\rarrow E$ belongs to $\cM$, by
Proposition~\ref{accessible-subcategory} the morphism
$(m'\circ c,\>e)$ factorizes as $t\rarrow v\rarrow m''$, where
$v\:U\rarrow V$ is some morphism belonging to~$\cN$.
 So we have a commutative diagram
$$
 \xymatrix{
  C \ar[r]^{m'} & D\ar[r]^{m''} & E \\
  & U \ar[u]_d \ar[r]^v & V \ar[u]_f \\
  S \ar[uu]^c \ar[rr]_t \ar[ru]_u
  && T \ar[u]_g \ar@/_2pc/[uu]_e
 }
$$
in $\sA$ with $\kappa$\+presentable objects $S$, $T$, $U$, and~$V$.

 Now we have a morphism $(c,d)\:u\rarrow m'$ in the category~$\sA^\to$,
where $u$~is a morphism with $\kappa$\+presentable domain and
codomain.
 Since the morphism $m'\:C\rarrow D$ belongs to $\cM$, by
Proposition~\ref{accessible-subcategory} the morphism $(c,d)$
factorizes as $u\rarrow n'\rarrow m'$, where $n'\:X\rarrow Y$
is some morphism belonging to~$\cN$.
 Hence we have a commutative diagram
\begin{equation} \label{common-diagram-for-two-proofs-I}
\begin{gathered}
 \xymatrix{
  C \ar[r]^{m'} & D\ar[r]^{m''} & E \\
  X \ar[r]^{n'} \ar[u]_{c'} & Y \ar[u]^{d'} \\
  & U \ar@/_1pc/[uu]_d \ar[r]^v \ar[u]^y & V \ar[uu]_f \\
  S \ar@/^2pc/[uuu]^c \ar[rr]_t \ar[ru]_u \ar[uu]_x
  && T \ar[u]_g \ar@/_2pc/[uuu]_e
 }
\end{gathered}
\end{equation}
in $\sA$ with $\kappa$\+presentable objects $S$, $T$, $U$, $V$, $X$,
and~$Y$.

 Finally, by condition~(i) for the class $\cN$ in $\sA_{<\kappa}$,
the span $v\:U\rarrow V$, \ $y\:U\rarrow Y$ has a pushout in
$\sA_{<\kappa}$, which by Lemma~\ref{small-colimits-preserved} is
also a pushout in~$\sA$.
 Denote the resulting pushout square by
$$
 \xymatrix{
  Y \ar@{..>}[r]^{n''} & Z \\
  U \ar[u]^y \ar[r]^v & V \ar@{..>}[u]^h
 }
$$
 Condition~(i) for the class $\cN$ in $\sA_{<\kappa}$ also tells us that
$n''\in\cN$ (since $v\in\cN$).

 We have a pair of morphisms $m''\circ d'\:Y\rarrow E$ and
$f\:V\rarrow E$ such that $m''\circ d'\circ y=f\circ v$.
 Hence there exists a unique morphism $e'\:Z\rarrow E$ making
the diagram
$$
 \xymatrix{
  D\ar[r]^{m''} & E \\
  Y \ar[r]^{n''} \ar[u]^{d'} & Z \ar@{..>}[u]^{e'} \\
  U \ar[u]^y \ar[r]^v & V \ar[u]^h \ar@/_1.5pc/[uu]_f
 }
$$
commutative.
 We have arrived to the commutative diagram
$$
 \xymatrix{
  C \ar[r]^{m'} & D\ar[r]^{m''} & E \\
  X \ar[r]^{n'} \ar[u]_{c'} & Y \ar[u]^{d'} \ar[r]^{n''}
  & Z \ar[u]^{e'} \\
  & U \ar[r]^v \ar[u]^y & V \ar[u]^h \\
  S \ar@/^2pc/[uuu]^c \ar[rr]_t \ar[ru]_u \ar[uu]_x
  && T \ar[u]_g \ar@/_2pc/[uuu]_e
 }
$$
proving that the morphism $t\rarrow m=m''\circ m'$ in $\sA^\to$
factorizes as $t\rarrow n\rarrow m$, where $n=n''\circ n'$.

 As both the morphisms $n'$ and~$n''$ belong to $\cN$ by
construction, so does their composition $n''\circ n'$, by
condition~(iii) for the class $\cN$ in~$\sA_{<\kappa}$.
 This finishes the proof of condition~(iii) for the class $\cM$
in~$\sA$.

 It remains to say that one obviously has $\cN\subset
\cM\cap\sA_{<\kappa}^\to$.
 In fact, by Proposition~\ref{accessible-subcategory},
\,$\cM\cap\sA_{<\kappa}^\to$ is precisely the class of all retracts of
the morphisms from~$\cN$ (the retracts being taken in the category
$\sA^\to$ or~$\sA_{<\kappa}^\to$).
 So the QE\+mono class $\cM$ in $\sA$ is locally $\kappa$\+coherent
by Lemmas~\ref{loc-coh-QE-mono-classes-directed-colimit-closed}
and~\ref{loc-coh-QE-mono-classes-trivial-characterization}.
\end{proof}

\begin{cor} \label{QE-mono-classes-bijection-cor}
 For any $\kappa$\+accessible category\/ $\sA$, there is a bijective
correspondence between locally $\kappa$\+coherent QE\+mono classes
in\/ $\sA$ and QE\+mono classes in the category\/~$\sA_{<\kappa}$
closed under retracts in\/~$\sA_{<\kappa}^\to$.
 The bijection assigns to every locally $\kappa$\+coherent QE\+mono
class $\cM$ in\/ $\sA$ the retraction-closed QE\+mono class
$\cN=\cM\cap\sA_{<\kappa}^\to$ in\/~$\sA_{<\kappa}$.
 Conversely, to every retraction-closed QE\+mono class $\cN$ in\/
$\sA_{<\kappa}$, the locally $\kappa$\+coherent QE\+mono class
$\cM=\varinjlim_{(\kappa)}\cN$ in\/ $\sA$ is assigned.
\end{cor}

\begin{proof}
 For every locally $\kappa$\+coherent QE\+mono class $\cM$ in $\sA$,
the class $\cN=\cM\cap\sA_{<\kappa}^\to$ is a retraction-closed
QE\+mono class in $\sA_{<\kappa}$ by
Proposition~\ref{restriction-of-loc-coh-to-presentables-is-QE-mono}.
 For every QE\+mono class $\cN$ in $\sA_{<\kappa}$, the class
$\cM=\varinjlim_{(\kappa)}\cN$ is a locally $\kappa$\+coherent
QE\+mono class in $\sA$ by
Theorem~\ref{direct-limit-closure-of-QE-mono-theorem}.
 For any locally $\kappa$\+coherent QE\+mono class $\cM$ in $\sA$,
one has $\cM=\varinjlim_{(\kappa)}(\cM\cap\sA_{<\kappa}^\to)$ by
Lemma~\ref{loc-coh-QE-mono-classes-trivial-characterization}.
 For any retraction-closed QE\+mono class $\cN$ in $\sA_{<\kappa}$,
one has $\cN=(\varinjlim_{(\kappa)}\cN)\cap\sA_{<\kappa}^\to$
by Proposition~\ref{accessible-subcategory}, as it was already
mentioned in the last paragraph of the proof of
Theorem~\ref{direct-limit-closure-of-QE-mono-theorem}.
\end{proof}

\begin{rem} \label{QE-mono-retraction-closedness-remark}
 For any $\kappa$\+accessible category $\sA$, the full subcategory
$\sA_{<\kappa}\subset\sA$ has split idempotents, because the category
$\sA$ has split idempotents by~\cite[Observation~2.4]{AR} and
$\sA_{<\kappa}$ is closed under retracts in~$\sA$.
 Conversely, for any small category $\sS$ with split idempotents and
any regular cardinal~$\kappa$, there exists a unique (up to natural
equivalence) $\kappa$\+accessible category $\sA$ such that the category
$\sA_{<\kappa}$ is equivalent to~$\sS$ \,\cite[Theorem~2.26
and Remark~2.26(1)]{AR}.

 A discussion of the retraction-closedness property of QE\+mono or
QE\+epi classes in additive categories, including in particular
idempotent-complete and weakly idempotent-complete additive categories,
can be found in~\cite[Theorems~1.1 and~1.2]{HR}.
 In the nonadditive context, we will continue this discussion below
in Section~\ref{strong-QE-epi-classes-secn}.
 At the moment, we restrict ourselves to the following simple
counterexample.
\end{rem}

\begin{ex} \label{QE-mono-not-retraction-closed-example}
 A QE\+mono class in an additive category $\sS$ with split idempotents
\emph{need not} be closed under retracts in general.
 Indeed, let $\sS=\k\vect$ be the category of finite-dimensional
vector spaces over a field~$\k$, and let $\cN$ be the class of
all monomorphisms~$n$ in $\sS$ with the dimension of the cokernel
$\dim_\k(\coker n)$ divisible by a fixed integer $q\ge2$.
 One can easily check that conditions~(i\+-iii) are satisfied for
the class $\cN$ in the category~$\sS$ (in particular, because
the cokernels are not changed by pushouts), but $\cN$ is \emph{not}
closed under retracts in the category~$\sS^\to$.
 Obviously, one has $\sS=\sA_{<\aleph_0}$, where $\sA=\k\Vect$ is
the finitely accessible category of $\k$\+vector spaces.
 The class $\cM=\varinjlim_{(\kappa)}\cN\subset\sA^\to$ consists of
all monomorphisms in~$\sA$.
\end{ex}

\Section{QE-Epi Classes}  \label{QE-epi-classes-secn}

 Let $\sC$ be a category.
 Dually to the discussion in Section~\ref{QE-mono-classes-secn},
a morphism $p\:D\rarrow E$ in $\sC$ is called a \emph{regular
epimorphism} if $p$~is the coequalizer of a parallel pair of morphisms
$d_1$, $d_2\:C\rightrightarrows D$.

 Clearly, every regular epimorphism is an epimorphism.
 Every split epimorphism is regular: if $s\:E\rarrow D$ is a morphism
such that $p\circ s=\id_E$, then $p$~is the coequalizer of the pair
of morphisms $s\circ p$ and $\id_D\:D\rightrightarrows D$.

 Dually to the definition in
Section~\ref{very-weak-cokernel-pairs-secn}, by the \emph{kernel pair}
of a morphism~$p$ one means the pullback of the cospan~$(p,p)$,
cf.\ diagram~\eqref{cospan-diagram}.
 A morphism $p\:D\rarrow E$ in $\sC$ is said to be an \emph{effective
epimorphism} if $p$~has a kernel pair $k_1$, $k_2\:K\rightrightarrows D$
and $p$~is the coequalizer of $(k_1,k_2)$.
 One can easily see that if a regular epimorphism~$p$ has a kernel
pair $(k_1,k_2)$, then $p$~is the coequalizer of $(k_1,k_2)$.
 So an epimorphism is effective if and only if it is regular
\emph{and} has a kernel pair.

\begin{lem} \label{effective-epimorphisms-lemma}
\textup{(a)} Let $C$ be a small category and\/ $\sA$ be a category
such that the colimits of all diagrams indexed by $C$ exist in\/~$\sA$.
 Let $P\:C\rarrow\sA^\to$ be a diagram such that $P(c)$ is
an effective epimorphism in\/ $\sA$ for all objects $c\in C$.
 Then the colimit of $P$, computed in\/ $\sA^\to$, is a regular
epimorphism in\/~$\sA$. \par
\textup{(b)} For any $\kappa$\+accessible category\/ $\sA$, the class
of effective epimorphisms is closed under $\kappa$\+directed colimits
in\/~$\sA^\to$.
\end{lem}

\begin{proof}
 Part~(a): the kernel pairs of the morphisms $P(c)$, \,$c\in C$, form
a diagram $K\:C\rarrow\sA^\rightrightarrows$ in the category
$\sA^\rightrightarrows$ of parallel pairs of morphisms in~$\sA$.
 By assumption, the morphism $P(c)$ is the coequalizer of the parallel
pair of morphisms $K(c)$ in $\sA$ for every object $c\in C$.
 The colimit of $K$ computed in $\sA^\rightrightarrows$, which exists
by assumption, is a parallel pair of morphisms in $\sA$ whose
coequalizer is the colimit of $P$ computed in~$\sA^\to$.
 Indeed, colimits commute with colimits in any category; so, in
particular, colimits indexed by $C$ preserve coequalizers in~$\sA$.

 Part~(b): by Lemma~\ref{directed-colimits-commute-with-small-limits},
in a $\kappa$\+accessible category, the $\kappa$\+directed colimits
commute with $\kappa$\+small limits; in particular,
the $\kappa$\+directed colimits preserve kernel pairs.
 So, if a morphism~$p$ is a $\kappa$\+directed colimit of effective
epimorphisms~$p_\xi$, then the colimit of the kernel pairs of~$p_\xi$
is the kernel pair of~$p$.
 Here we are assuming that the index~$\xi$ ranges over
a $\kappa$\+directed poset~$\Xi$.
 In any category, colimits commute with colimits; in particular,
the $\kappa$\+directed colimits preserve coequalizers.
 Thus $p$~is the coequalizer of its kernel pair.
\end{proof}

 Let $\cP$ be a class of morphism in a category~$\sC$.
 We will say that $\cP$ is a \emph{QE\+epi class} in $\sC$ if
the following conditions are satisfied:
\begin{enumerate}
\renewcommand{\theenumi}{\roman{enumi}$^*$}
\item All pullbacks of all morphisms from $\cP$ exists in $\sC$,
and the class $\cP$ is stable under pulbacks.
 In other words, for any cospan $p\:D\rarrow E$, \ $f\:E'\rarrow E$
such that $p\in\cP$, the pullback $D'$ exists, and the morphism
$p'\:D'\rarrow E'$ belongs to~$\cP$,
\begin{equation} \label{QE-epi-defn-pullback-diagram}
\begin{gathered}
 \xymatrix{
  D \ar[r]^{p} & E \\
  D' \ar@{..>}[u]^{f'} \ar@{..>}[r]_{p'} & E' \ar[u]_f
 }
\end{gathered}
\end{equation}
\item In particular, condition~(i$^*$) implies that all morphisms from
$\cP$ have kernel pairs in~$\sC$.
 It is further required that every morphism from $\cP$ is
the coequalizer of its kernel pair.
 In other words, all the morphisms from $\cP$ must be effective
epimorphisms.
\item All the identity morphisms in $\sC$ belong to $\cP$, and the class
of morphisms $\cP$ is closed under compositions.
\end{enumerate}

 It is clear from the preceding discussion that, assuming
condition~(i$^*$), condition~(ii$^*$) is equivalent to the condition
that all morphisms from $\cP$ are regular epimorphisms in~$\sC$.

\begin{exs} \label{QE-epi-classes-examples}
 (1)~If $\sC$ is an additive category, then a QE\+epi class of
morphisms in $\sC$ is the same thing as a structure of \emph{left
exact category} on $\sC$ in the sense of~\cite[Definition~3.1]{BC}.
 In the terminology of~\cite[Definition~2.2]{HR}, such additive
categories with an additional structure are called
\emph{deflation-exact categories}.

\smallskip
 (2)~In any regular category $\sC$ (in the sense of~\cite{BGO,Gra}),
the class of all regular epimorphisms is a QE\+epi
class~\cite[Definition~1.10 and Proposition~1.13(3)]{Gra}.

\smallskip
 (3)~The compositions of split epimorphisms are always split
epimorphisms, and all split epimorphisms are regular
(as explained above).
 Furthermore, all pushouts of split epimorphisms are split
epimorphisms by the argument dual to the one in
Example~\ref{QE-mono-classes-examples}(3).
 Therefore, the class $\cP$ of all split epimorphisms in
a category $\sC$ is a QE\+epi class if and only all pullbacks of
split epimorphisms exist in~$\sC$.
\end{exs}

\begin{rem} \label{pullbacks-of-split-epimorphisms-remark}
 Dually to Remark~\ref{pushouts-of-split-monomorphisms-remark},
all pullbacks of split epimorphisms exist in an additive category
if and only if the category is weakly idempotent-complete.
 However, the idempotent-complete preadditive category $\sA$ of
$\k$\+vector spaces of finite dimension~$\le n$ (where $n\ge1$ is
a fixed integer) does not have pullbacks of split epimorphisms.
 In fact, the category $\sA$ does not even have very weak split
pullbacks of split epimorphisms; see
Example~\ref{no-very-weak-split-pullback-pair-example}.
 It is also clear from
Example~\ref{no-very-weak-split-pullback-pair-example}
(take $j=n$) that the category $\sA$ does not have (even weak) kernel
pairs of split epimorphisms.
\end{rem}

 Let $\cP$ be a QE\+epi class in a category~$\sC$.
 For every morphism $p\:D\rarrow E$ belonging to $\cP$, consider its
kernel pair $k_1$, $k_2\:K\rightrightarrows D$.
 By a \emph{$\cP$\+sequence} we mean a diagram
$$
 \xymatrix{
  K \ar@<2pt>[r]^{k_1}\ar@<-2pt>[r]_{k_2} & D \ar[r]^p & E
 } 
$$
arising from some morphism $p\in\cP$ in this way.

 Let $\sA$ be a $\kappa$\+accessible category.
 We will say that a QE\+epi class $\cP$ in $\sA$ is \emph{locally
$\kappa$\+coherent} if the $\cP$\+sequences in $\sA$ are precisely all
the $\kappa$\+directed colimits of $\cP$\+sequences in $\sA$ with
all the three terms $K$, $D$, $E$ belonging to~$\sA_{<\kappa}$.
 Here the $\kappa$\+directed colimits are taken in the diagram
category~$\sA^{J^\sop}$, where $J$ is the finite category defined
in Section~\ref{QE-mono-classes-secn}.
 The terminology ``locally $\kappa$\+coherent'' comes from
the paper~\cite[Section~1]{Plce}.

\begin{lem} \label{loc-coh-QE-epi-classes-directed-colimit-closed}
 For any locally $\kappa$\+coherent QE\+epi class $\cP$ in
a $\kappa$\+accessible category\/ $\sA$, the class of all
$\cP$\+sequences is closed under $\kappa$\+directed colimits
in\/~$\sA^{J^\sop}$.
\end{lem}

\begin{proof}
 This is completely similar to
Lemma~\ref{loc-coh-QE-mono-classes-directed-colimit-closed}.
\end{proof}

 Notice the difference between the formulations of the following lemma
and its version for QE\+mono classes
(Lemma~\ref{loc-coh-QE-mono-classes-trivial-characterization} above).

\begin{lem} \label{loc-coh-QE-epi-classes-trivial-characterization}
 A QE\+epi class $\cP$ in a $\kappa$\+accessible category\/ $\sA$ is
locally $\kappa$\+coherent if and only if the following two conditions
hold:
\begin{enumerate}
\item $\cP$ is precisely the class of all $\kappa$\+directed colimits
of the morphisms from $\cP$ whose domains and codomains are
$\kappa$\+presentable.
 Here the $\kappa$\+directed colimits are taken in the category of
morphisms\/~$\sA^\to$.
\item For any morphism $p\in\cP$ whose domain and codomain are
$\kappa$\+presentable, the domain of the kernel pair of~$p$ is
$\kappa$\+presentable as well.
\end{enumerate}
\end{lem}

\begin{proof}
 The proof of this lemma only uses condition~(i$^*$) from the definition
of a QE\+epi class; conditions~(ii$^*$) and~(iii$^*$) are not used.
 Let $\cP$ be a locally $\kappa$\+coherent QE\+epi class in~$\sA$.
 To check condition~(2), assume that a morphism~$p$ in $\sA$ is
a $\kappa$\+directed colimit of some morphisms~$p_\xi$, the colimit
being taken in the category $\sA^\to$, that is
$p=\varinjlim_{\xi\in\Xi}^{\sA^\to}p_\xi$, where $\Xi$ is
a $\kappa$\+directed poset.
 If the domain and codomain of~$p$ are $\kappa$\+presentable, then
$p$~is a $\kappa$\+presentable object of $\sA^\to$, and it follows
that $p$~is a retract of one of the morphisms~$p_\xi$.
 Now if $p$~has a kernel pair $(k_1,k_2)$ in $\sA$ and $p_\xi$~has
a kernel pair $(k_{\xi,1},k_{\xi,2})$ in $\sA$, then $(k_1,k_2)$
is a retract of $(k_{\xi,1},k_{\xi,2})$.
 In particular, the domain $K$ of $(k_1,k_2)$ is a retract of
the domain $K_\xi$ of $(k_{\xi,1},k_{\xi,2})$.
 Hence if $K_\xi$ is $\kappa$\+presentable, then $K$ is
$\kappa$\+presentable, too.
 After condition~(2) is proved, condition~(1) becomes obvious.
 Conversely, if conditions (1) and~(2) hold then, in order to check
that $\cP$ is locally $\kappa$\+coherent, one needs to use the fact
that $\kappa$\+directed colimits preserve finite limits (in particular,
kernel pairs) in~$\sA$.
 This is Lemma~\ref{directed-colimits-commute-with-small-limits}.
\end{proof}

 Let $\sA$ be a $\kappa$\+accessible category.
 We will say that a locally $\kappa$\+coherent QE\+epi class $\cP$ in
$\sA$ is \emph{strongly locally $\kappa$\+coherent} if it satisfies
the following stronger version of condition~(2) from
Lemma~\ref{loc-coh-QE-epi-classes-trivial-characterization}:
\begin{enumerate}
\renewcommand{\theenumi}{\arabic{enumi}$'$}
\setcounter{enumi}{1}
\item for any pullback diagram~\eqref{QE-epi-defn-pullback-diagram}
in $\sA$ with $\kappa$\+presentable objects $D$, $E$, $E'$ and
a morphism $p\in\cP$, the object $D'$ is also $\kappa$\+presentable.
\end{enumerate}

\begin{rem}
 Under a natural additional assumption, any locally $\kappa$\+coherent
QE\+epi class in a $\kappa$\+accessible \emph{additive} category $\sA$
is strongly locally $\kappa$\+coherent.
 Indeed, given a morphism $p\:D\rarrow E$ with a kernel pair
$k_1$, $k_2\:K\rightrightarrows D$ in an idempotent-complete additive
category $\sA$, the kernel $k'\:K'\rarrow D$ of the morphism~$p$ can be
constructed as the image of a suitable idempotent endomorphism
$K\rarrow K$.
 In fact, one has $K\simeq D\oplus K'$.
 Given a morphism $f\:E'\rarrow E$, the pullback $D'$ of the pair
of morphisms $p\:D\rarrow E$ and $f\:E'\rarrow E$ can be constructed as
the kernel of the induced morphism $(p,f)\:D\oplus E'\rarrow E$.
 Hence, in the situation at hand, condition~(2) implies
$D'\in\sA_{<\kappa}$ provided that $D$, $E$, $E'\in\sA_{<\kappa}$ and
$(p,f)\in\cP$.

 It remains to make sure that $(p,f)\in\cP$ whenever $p\in\cP$.
 Denoting by $i_D\:D\rarrow D\oplus E'$ the direct summand injection,
we have $(p,f)\circ i_D=p$.
 By the pullback axiom~(i$^*$) above, the pullback $D'$ of
the pair of morphisms $p$ and~$f$ exists in~$\sA$; so
the morphism~$(p,f)$ has a kernel in~$\sA$.
 Assuming the axiom dual to~\cite[axiom~{[R3]} from Definition~3.2]{BC},
or which is the same, \cite[axiom~R3 from Definition~2.3]{HR}, it
follows that $(p,f)\in\cP$ whenever $p\in\cP$.
 See Section~\ref{strong-QE-epi-classes-secn} below for a further
discussion.

 We are not aware of any example of a locally $\kappa$\+coherent
QE\+epi class (in any $\kappa$\+accessible category)
that is not strongly locally $\kappa$\+coherent.
\end{rem}

\begin{prop} \label{restriction-of-loc-coh-to-presentables-is-QE-epi}
 Let $\cP$ be a locally $\kappa$\+coherent QE\+epi class in
a $\kappa$\+accessible category\/~$\sA$.
 Then the class $\cP\cap\sA_{<\kappa}^\to$ of all morphisms from $\cP$
with $\kappa$\+presentable domains and codomains is closed under
retracts in the category\/~$\sA_{<\kappa}^\to$.
 The locally $\kappa$\+coherent QE\+epi class $\cP$ is strongly locally
$\kappa$\+coherent if and only if the class $\cP\cap\sA_{<\kappa}^\to$
is a QE\+epi class in the category\/~$\sA_{<\kappa}$.
\end{prop}

\begin{proof}
 The first assertion is provable similarly to the proof of
the second assertion of
Proposition~\ref{restriction-of-loc-coh-to-presentables-is-QE-mono}.
 Let us prove the second assertion.
 ``If'': assuming that the class $\cP\cap\sA_{<\kappa}^\to$
satisfies condition~(i$^*$) in the category $\sA_{<\kappa}$,
condition~(2$'$) for the class $\cP$ in $\sA$ follows, because all
limits (in particular, pullbacks) that exist in $\sA_{<\kappa}$ are
also limits in $\sA$ by Lemma~\ref{limits-preserved}.
 ``Only if'': condition~(iii$^*$) obviously holds for
$\cP\cap\sA_{<\kappa}^\to$ whenever it holds for~$\cP$.
 Condition~(i$^*$) holds for $\cP\cap\sA_{<\kappa}^\to$ whenever it
holds for $\cP$ \emph{and} $\cP$ is strongly locally $\kappa$\+coherent,
because any square diagram in $\sA_{<\kappa}$ that is a pullback
diagram in $\sA$ is also a pullback diagram in~$\sA_{<\kappa}$.
 Finally, condition~(ii$^*$) holds for $\cP\cap\sA_{<\kappa}^\to$ 
whenever conditions (i$^*$) and~(ii$^*$) hold for $\cP$ and $\cP$
is strongly locally $\kappa$\+coherent, because any diagram in
$\sA_{<\kappa}$ that is a coequalizer diagram in $\sA$ is also
a coequalizer diagram in~$\sA_{<\kappa}$.
\end{proof}

\Section{Construction of Strongly Locally Coherent QE-Epi Classes}
\label{QE-epi-construction-secn}

 Let $D$ be the finite category defined in
Section~\ref{QE-mono-construction-secn}.
 By \emph{the category of cospans} in a category $\sC$ we mean
the category of $D^\sop$\+shaped diagrams in $\sC$, i.~e., the category
of functors~$\sC^{D^\sop}$.

\begin{thm} \label{direct-limit-closure-of-QE-epi-theorem}
 Let\/ $\sA$ be a $\kappa$\+accessible category and $\cQ$ be
a QE\+epi class in the category\/~$\sA_{<\kappa}$.
 Then the class\/ $\varinjlim_{(\kappa)}\cQ\subset\sA^\to$ of all
$\kappa$\+directed colimits of morphisms from $\cQ$ (the colimits
being taken in\/~$\sA^\to$) is a strongly locally $\kappa$\+coherent
QE\+epi class in the category\/~$\sA$.
\end{thm}

\begin{proof}
 The argument is largely similar to the proof of
Theorem~\ref{direct-limit-closure-of-QE-mono-theorem}, but there
are some differences.
 Put $\cP=\varinjlim_{(\kappa)}\cQ$.
 In order to check that condition~(i$^*$) for the class $\cQ$ in
$\sA_{<\kappa}$ implies condition~(i$^*$) for the class $\cP$ in~$\sA$,
let us show that all cospans $(p,f)$ in $\sA$ (where $p\:D\rarrow E$
and $f\:E'\rarrow E$) such that $p\in\cP$, are $\kappa$\+directed
colimits, in the category of cospans in~$\sA$, of cospans $(q,g)$
in $\sA_{<\kappa}$ such that $q\in\cQ$.

 As usual, we keep in mind the fact that the category of morphisms
$\sA^\to$ is $\kappa$\+accessible, and its $\kappa$\+presentable objects
are precisely all the morphisms with $\kappa$\+presentable domains and
codomains.
 Let $\sP$ be the full subcategory in $\sA^\to$ whose objects are all
the morphisms belonging to~$\cP$.
 By Proposition~\ref{accessible-subcategory} applied to the category
$\sA^\to$, the category $\sP$ is $\kappa$\+accessible, and its
$\kappa$\+presentable objects are precisely all the retracts of
the morphisms belonging to~$\cQ$.
 Now let $F\:\sA\rarrow\sA$ be the identity functor, and let
$G\:\sP\rarrow\sA$ be the functor taking every morphism $p\:D\rarrow E$
to its codomain~$E$.
 Then the category $F\down G$ defined in
Section~\ref{preliminaries-secn} is precisely the category of all
cospans $(p,f)$ in $\sA$ with $p\in\cP$.
 By Proposition~\ref{comma-category-accessible}, it follows that
the category $F\down G$ is $\kappa$\+accessible, and its
$\kappa$\+presentable objects are precisely all the cospans $(q',g')$
in $\sA_{<\kappa}$ such that $q'$~is a retract of a morphism
belonging to~$\cQ$.

 Moreover, let $\sS\subset\sA_{<\kappa}$ be a set of representatives of 
the isomorphism classes of $\kappa$\+presentable objects of $\sA$, and
let $\sT\subset\sP_{<\kappa}$ be a set of representatives of
the isomorphism classes of morphisms belonging to~$\cQ$.
 Then Proposition~\ref{comma-category-objects-directed-colimits}
applied to the sets of $\kappa$\+presentable objects $\sS$ and $\sT$
tells us that all cospans $(p,f)\in\sA^{D^\sop}$ with $p\in\cP$ are
$\kappa$\+directed colimits, in the category of cospans~$\sA^{D^\sop}$,
of cospans $(q,g)\in(\sA_{<\kappa})^{D^\sop}$ with $q\in\cQ$
(as desired).

 By Lemma~\ref{limits-preserved}, any limit in $\sA_{<\kappa}$ is also
a limit in~$\sA$.
 In particular, this applies to pullbacks.
 So pullbacks of the cospans $(q,g)$ in $\sA_{<\kappa}$ remain pullbacks
in~$\sA$.
 Since $\kappa$\+directed colimits in $\sA$ preserve pullbacks
by Lemma~\ref{directed-colimits-commute-with-small-limits}, we
have shown that condition~(i$^*$) for $\cQ$ in $\sA_{<\kappa}$ implies
condition~(i$^*$) for $\cP$ in~$\sA$.

 The proof of the assertion that conditions~(i$^*$) and~(ii$^*$) for
$\cQ$ in $\sA_{<\kappa}$ imply condition~(ii$^*$) for $\cP$ in $\sA$ is
somewhat similar, and based on the arguments above together with
Lemmas~\ref{small-colimits-preserved}
and~\ref{effective-epimorphisms-lemma}.
 By assumption, every morphism $q\in\cQ$ is the coequalizer in
$\sA_{<\kappa}$ of its kernel pair in~$\sA_{<\kappa}$.
 We have already seen that the kernel pair of~$q$ in~$\sA_{<\kappa}$
is also the kernel pair of~$q$ in~$\sA$.
 By the second assertion of Lemma~\ref{small-colimits-preserved}, any
finite colimit that exists in $\sA_{<\kappa}$ is also a colimit
in~$\sA$.
 Therefore, $q$~is the coequalizer in $\sA$ of its kernel pair in~$\sA$.
 So $q$~is an effective epimorphism in~$\sA$.
 Now any morphism $p\in\cP$ is a $\kappa$\+directed colimit in $\sA^\to$
of morphisms $q\in\cQ$, and $\kappa$\+directed colimits of effective
epimorphisms in $\sA$ are effective epimorphisms in $\sA$ by
Lemma~\ref{effective-epimorphisms-lemma}(b).

 Let us prove that conditions~(i$^*$) and~(iii$^*$) for $\cQ$ in
$\sA_{<\kappa}$ imply condition~(iii$^*$) for $\cP$ in~$\sA$.
 The assertion concerning the identity morphisms is obvious.
 We need to show that the class $\cP$ is closed under compositions.
 Let $p'\:C\rarrow D$ and $p''\:D\rarrow E$ be two morphisms belonging
to~$\cP$.
 By Proposition~\ref{accessible-subcategory}, in order to prove that
the composition $p=p''\circ p'$ belongs to $\cP$, we need to check that
every morphism $t\rarrow p$ in $\sA^\to$ from a morphism~$t$ with
$\kappa$\+presentable domain and codomain, $t\in(\sA^\to)_{<\kappa}
=(\sA_{<\kappa})^\to$, into the morphism~$p$ factorizes as
$t\rarrow q\rarrow p$, where $q\in\cQ$.
 So we have a commutative diagram
$$
 \xymatrix{
  C \ar[r]^{p'} & D\ar[r]^{p''} & E \\
  T \ar[u]^c \ar[rr]_t && S \ar[u]_e
 }
$$
in $\sA$ with $\kappa$\+presentable objects $S$ and~$T$.

 Arguing exactly as in the proof of
Theorem~\ref{direct-limit-closure-of-QE-mono-theorem}, we construct
a commutative diagram similar
to~\eqref{common-diagram-for-two-proofs-I}.
 Let us redraw it here in our current notation:
\begin{equation} \label{common-diagram-for-two-proofs-II}
\begin{gathered}
 \xymatrix{
  C \ar[r]^{p'} & D\ar[r]^{p''} & E \\
  U \ar[r]^{v} \ar[u]_{c'} & V \ar[u]^{d'} \\
  & Y \ar@/_1pc/[uu]_d \ar[r]^{q''} \ar[u]^h & Z \ar[uu]_f \\
  T \ar@/^2pc/[uuu]^c \ar[rr]_t \ar[ru]_y \ar[uu]_u
  && S \ar[u]_g \ar@/_2pc/[uuu]_e
 }
\end{gathered}
\end{equation}
 So~\eqref{common-diagram-for-two-proofs-II} is a commutative diagram
in $\sA$ with $\kappa$\+presentable objects $S$, $T$, $U$, $V$, $X$,
and~$Y$, and morphisms $v$, $q''\in\cQ$.

 By condition~(i$^*$) for the class $\cQ$ in $\sA_{<\kappa}$,
the cospan $v\:U\rarrow V$, \ $h\:Y\rarrow V$ has a pullback in
$\sA_{<\kappa}$, which by Lemma~\ref{limits-preserved} is
also a pullback in~$\sA$.
 Denote the resulting pullback square by
$$
 \xymatrix{
  U \ar[r]^v & V \\
  X \ar@{..>}[u]_{h'} \ar@{..>}[r]^{q'} & Y \ar[u]^h
 }
$$
 Condition~(i$^*$) for the class $\cQ$ in $\sA_{<\kappa}$ also tells us
that $q'\in\cQ$ (since $v\in\cQ$).

 We have a pair of morphisms $u\:T\rarrow U$ and
$y\:T\rarrow Y$ such that $v\circ u=h\circ y$.
 Hence there exists a unique morphism $x\:T\rarrow X$ making
the diagram
$$
 \xymatrix{
  U\ar[r]^v & V \\
  X \ar[r]^{q'} \ar[u]_{h'} & Y \ar[u]^h \\
  T \ar@{..>}[u]^x \ar@/^1.5pc/[uu]^u \ar[ru]_y
 }
$$
commutative.
 We have arrived to the commutative diagram
$$
 \xymatrix{
  C \ar[r]^{p'} & D\ar[r]^{p''} & E \\
  U \ar[r]^v \ar[u]_{c'} & V \ar[u]^{d'} \\
  X \ar[r]^{q'} \ar[u]_{h'} & Y \ar[r]^{q''} \ar[u]^h & Z \ar[uu]^f \\
  T \ar@/^2pc/[uuu]^c \ar[rr]_t \ar[ru]_y \ar[u]^x
  && S \ar[u]_g \ar@/_2pc/[uuu]_e
 }
$$
proving that the morphism $t\rarrow p=p''\circ p'$ in $\sA^\to$
factorizes as $t\rarrow q\rarrow p$, where $q=q''\circ q'$.

 As both the morphisms $q'$ and~$q''$ belong to $\cQ$ by
construction, so does their composition $q''\circ q'$, by
condition~(iii$^*$) for the class $\cQ$ in~$\sA_{<\kappa}$.
 This finishes the proof of condition~(iii$^*$) for the class $\cP$
in~$\sA$.

 It remains to prove that the QE\+epi class $\cP$ is strongly locally
$\kappa$\+coherent.
 In fact, by Proposition~\ref{accessible-subcategory},
\,$\cP\cap\sA_{<\kappa}^\to$ is precisely the class of all retracts of
the morphisms from~$\cQ$ (the retracts being taken in the category
$\sA^\to$ or~$\sA_{<\kappa}^\to$).
 So the QE\+epi class $\cP$ in $\sA$ satisfies condition~(1) of
Lemma~\ref{loc-coh-QE-epi-classes-trivial-characterization} in view of
Lemma~\ref{loc-coh-QE-epi-classes-directed-colimit-closed}.
 To check condition~(2$'$) from Section~\ref{QE-epi-classes-secn}
for the class $\cP$, notice that, in view of the proof of
condition~(i$^*$) for the class $\cP$ above, every cospan
$(p,f)\in(\sA_{<\kappa})^{D^\sop}$ in $\sA_{<\kappa}$ with a morphism
$p\in\cP$ is a retract of a cospan $(q,g)\in(\sA_{<\kappa})^{D^\sop}$
in $\sA_{<\kappa}$ with a morphism $q\in\cQ$.
 Furthermore, the pullback of $(q,g)$ in $\sA_{<\kappa}$ is also
the pullback of $(q,g)$ in~$\sA$.
 It follows that the pullback of $(p,f)$ in $\sA$ is a retract of
the pullback of $(q,g)$, and it remains to point out that the full
subcategory $\sA_{<\kappa}\subset\sA$ is closed under retracts in~$\sA$.
\end{proof}

\begin{cor} \label{QE-epi-classes-bijection-cor}
 For any $\kappa$\+accessible category\/ $\sA$, there is a bijective
correspondence between strongly locally $\kappa$\+coherent QE\+epi 
classes in\/ $\sA$ and QE\+epi classes in the category\/~$\sA_{<\kappa}$
closed under retracts in\/~$\sA_{<\kappa}^\to$.
 The bijection assigns to every strongly locally $\kappa$\+coherent
QE\+epi class $\cP$ in\/ $\sA$ the retraction-closed QE\+epi class
$\cQ=\cP\cap\sA_{<\kappa}^\to$ in\/~$\sA_{<\kappa}$.
 Conversely, to every retraction-closed QE\+epi class $\cQ$ in\/
$\sA_{<\kappa}$, the strongly locally $\kappa$\+coherent QE\+epi class
$\cP=\varinjlim_{(\kappa)}\cQ$ in\/ $\sA$ is assigned.
\end{cor}

\begin{proof}
 For every strongly locally $\kappa$\+coherent QE\+epi class $\cP$
in $\sA$, the class $\cQ=\cP\cap\sA_{<\kappa}^\to$ is
a retraction-closed QE\+epi class in $\sA_{<\kappa}$ by
Proposition~\ref{restriction-of-loc-coh-to-presentables-is-QE-epi}.
 For every QE\+epi class $\cQ$ in $\sA_{<\kappa}$, the class
$\cP=\varinjlim_{(\kappa)}\cQ$ is a strongly locally $\kappa$\+coherent
QE\+epi class in $\sA$ by
Theorem~\ref{direct-limit-closure-of-QE-epi-theorem}.
 For any strongly locally $\kappa$\+coherent QE\+epi class $\cP$
in $\sA$, one has $\cP=\varinjlim_{(\kappa)}(\cP\cap\sA_{<\kappa}^\to)$
by Lemma~\ref{loc-coh-QE-epi-classes-trivial-characterization}(1).
 For any retraction-closed QE\+epi class $\cQ$ in $\sA_{<\kappa}$,
one has $\cQ=(\varinjlim_{(\kappa)}\cQ)\cap\sA_{<\kappa}^\to$
by Proposition~\ref{accessible-subcategory}, as it was already
mentioned in the last paragraph of the proof of
Theorem~\ref{direct-limit-closure-of-QE-epi-theorem}.
\end{proof}

\begin{ex} \label{QE-epi-not-retraction-closed-example}
 Dually to Example~\ref{QE-mono-not-retraction-closed-example},
a QE\+epi class in an additive category $\sS$ with split idempotents
\emph{need not} be closed under retracts in general.
 Indeed, let $\sS=\k\vect$ be the category of finite-dimensional
vector spaces over a field~$\k$, and let $\cQ$ be the class of
all epimorphisms~$q$ in $\sS$ with the dimension of the kernel
$\dim_\k(\ker q)$ divisible by a fixed integer $n\ge2$.
 Then conditions~(i$^*$\+-iii$^*$) are satisfied for
the class $\cQ$ in the category~$\sS$, but $\cQ$ is \emph{not}
closed under retracts in the category~$\sS^\to$.
 Furthermore, one has $\sS=\sA_{<\aleph_0}$, where $\sA=\k\Vect$ is
the finitely accessible category of $\k$\+vector spaces.
 The class $\cP=\varinjlim_{(\kappa)}\cQ\subset\sA^\to$ consists of
all epimorphisms in~$\sA$.
\end{ex}

\Section{Strong QE-Epi Classes}  \label{strong-QE-epi-classes-secn}

 Let $\sC$ be a category.
 We will say that a QE\+epi class $\cP$ in~$\sC$ (as defined in
Section~\ref{QE-epi-classes-secn}) is a \emph{strong QE\+epi class}
if it satisfies the following additional condition:
\begin{enumerate}
\renewcommand{\theenumi}{\roman{enumi}$^*$}
\setcounter{enumi}{3}
\item If $p$, $q$ is a composable pair of morphisms in $\sC$ and
$p\circ q\in\cP$, then $p\in\cP$.
\end{enumerate}
 In the context of additive categories, our axiom~(iv$^*$) coincides
with~\cite[axiom~R3$^+$ from Definition~3.1]{HR}.
 A similar but slightly more general definition of a \emph{strongly left
exact} additive category can be found in~\cite[Definition~3.2]{BC}.

\begin{lem} \label{strong-QE-epi-classes-closed-under-retracts}
 Any strong QE\+epi class in a category\/ $\sC$ is closed under
retracts in the category\/~$\sC^\to$.
\end{lem}

\begin{proof}
 Let $p\:B\rarrow C$ be a morphism belonging to~$\cP$, and let
$q\:D\rarrow E$ be a retract of the morphism~$p$.
 So we have a commutative diagram
$$
 \xymatrix{
  D \ar[r]^q & E \\
  B \ar[r]^p \ar[u]_{s_D} & C \ar[u]^{s_E} \\
  D \ar[r]^q \ar[u]_{i_D} \ar@/^2pc/[uu]^{\id_D}
  & E \ar[u]^{i_E} \ar@/_2pc/[uu]_{\id_E}
 }
$$
 By condition~(i$^*$), there exists a pullback diagram
$$
 \xymatrix{
  B \ar[r]^p & C \\
  F \ar@{..>}[r]^{p'} \ar@{..>}[u]_j & E \ar[u]^{i_E}
 }
$$
in the category $\sC$, and the morphism $p'\:F\rarrow E$ belongs
to~$\cP$.
 Now we have commutative diagrams
$$
 \xymatrix{
  D \ar[r]^q & E \\
  B \ar[r]^p \ar[u]_{s_D} & C \ar[u]^{s_E} \\
  F \ar[r]^{p'} \ar[u]_j & E \ar[u]^{i_E} \ar@/_2pc/[uu]_{\id_E}
 }
 \qquad\qquad
 \xymatrix{
  D \ar[r]^q & E \\
  B \ar[u]^{s_D} \\
  F \ar[u]^j \ar[ruu]_{p'}
 }
$$
with $p'\in\cP$.
 By condition~(iv$^*$), it follows that $q\in\cP$.
\end{proof}

 The following proposition is a weak nonadditive version
of~\cite[Theorem~1.2]{HR}.

\begin{prop} \label{strong-QE-epi-classes-characterized}
 Let\/ $\sC$ be a category with finite coproducts and $\cP$ be
a QE\+epi class in\/~$\sC$.
 Then $\cP$ is a strong QE\+epi class in\/ $\sC$ if and only if
the following two conditions hold:
\begin{enumerate}
\item the class $\cP$ is closed under retracts in\/~$\sC^\to$;
\item for any three objects $A$, $B$, $C\in\sC$ and any two
morphisms $p\:A\rarrow C$ and $f\:B\rarrow C$ such that $p\in\cP$,
the induced morphism from the coproduct $(p,f)\:A\sqcup B\rarrow C$
also belongs to~$\cP$.
\end{enumerate}
\end{prop}

\begin{proof}
 ``Only if'': condition~(1) holds by
Lemma~\ref{strong-QE-epi-classes-closed-under-retracts} (this
implication does not depend on the assumption of existence of
finite coproducts in~$\sC$).
 Condition~(2) follows from~(iv$^*$), since the composition
$A\rarrow A\sqcup B\rarrow C$ is equal to $p\in\cP$.

 ``If'': Let $q\:C\rarrow D$ and $p\:D\rarrow E$ be a pair of morphisms
in $\sC$ such that the composition $r=p\circ q$ belongs to~$\cP$.
 Then, by~(2), the morphism $(r,p)\:C\sqcup D\rarrow E$ belongs
to~$\cP$.
 It remains to point out that the morphism~$p$ is a retract of
the morphism~$(r,p)$, in view of commutativity of the diagram
$$
 \xymatrix{
  D \ar[r]^p & E \\
  C\sqcup D \ar[u]^{(q,\id_D)} \ar[r]^-{(r,p)} & E \ar@{=}[u] \\
  D \ar[u]^{i_D} \ar[r]^p & E \ar@{=}[u]
 }
$$
where $i_D\:D\rarrow C\sqcup D$ is the coproduct injection.
\end{proof}

\Section{Characterization of Strongly Locally Coherent
Strong~QE-Epi~Class}

 Let $\sA$ be a $\kappa$\+accessible category.
 By a strongly locally $\kappa$\+coherent strong QE\+epi class
in $\sA$ we mean a strong QE\+epi class $\cP$ in~$\sA$
(in the sense of Section~\ref{strong-QE-epi-classes-secn}) that is
strongly locally $\kappa$\+coherent as a QE\+epi class in~$\sA$
(in the sense of Section~\ref{QE-epi-classes-secn}).

\begin{prop} \label{slc-QE-epi-class-characterization-prop}
 Let\/ $\sA$ be a $\kappa$\+accessible category with finite coproducts
and $\cQ$ be a strong QE\+epi class in the category\/~$\sA_{<\kappa}$.
 Let $\cP=\varprojlim_{(\kappa)}\cQ\subset\sA^\to$ be the related
(strongly locally $\kappa$\+coherent) QE\+epi class in the category\/
$\sA$, as per
Theorem~\ref{direct-limit-closure-of-QE-epi-theorem}.
 Then a morphism $p\:D\rarrow E$ in\/ $\sA$ belongs to $\cP$ if and
only if, for every object $S\in\sA_{<\kappa}$ and every morphism
$e\:S\rarrow E$ there exists a commutative square diagram
\begin{equation} \label{slc-QE-epi-class-characterization-diagram}
\begin{gathered}
 \xymatrix{
  D \ar[r]^p & E \\
  T \ar@{..>}[u]^d \ar@{..>}[r]_q & S \ar[u]_e
 }
\end{gathered}
\end{equation}
with an object $T\in\sA_{<\kappa}$ and a morphism $q\in\cQ$.
\end{prop}

\begin{proof}
 This is a nonadditive version of~\cite[Lemmas~1.5 and~2.3]{Plce}.
 ``Only if'': this implication does not depend on the assumption of
existence of finite coproducts in~$\sA$.
 The assumption that $\cQ$ is a strong QE\+epi class is not needed
for this implication, either; it is only important that the class
$\cQ$ satisfies condition~(i$^*$) in the category~$\sA_{<\kappa}$.

 Suppose that $p=\varinjlim_{\xi\in\Xi}^{\sA^\to}u_\xi$, where
$u_\xi\:X_\xi\rarrow Y_\xi$ are some morphisms belonging to $\cQ$
and $\Xi$ is a $\kappa$\+directed poset.
 Let $x_\xi\:X_\xi\rarrow D$ and $y_\xi\:Y_\xi\rarrow E$ be
the natural morphisms to the colimit.
 Then, since the object $S$ is $\kappa$\+accessible, the morphism
$e\:S\rarrow E$ factorizes through the morphism~$y_\xi$ for some
index $\xi\in\Xi$.
 So we have a morphism $v\:S\rarrow Y_\xi$ such that $e=y_\xi\circ v$.
 Put $X=X_\xi$, \ $Y=Y_\xi$, \ $x=x_\xi$, \ $y=y_\xi$, and $u=u_\xi$.

 Applying condition~(i$^*$) to the pair of morphisms
$u\:X\rarrow Y$ and $v\:S\rarrow Y$ in the category $\sA_{<\kappa}$,
with the morphism~$u$ belonging to $\cQ$, we obtain a pullback
diagram
$$
 \xymatrix{
  X \ar[r]^u & Y \\
  T \ar@{..>}[u]^{v'} \ar@{..>}[r]_q & S \ar[u]^v
 }
$$
in the category~$\sA_{<\kappa}$ with the morphism~$q$ belonging
to~$\cQ$.
 Now we have the commutative diagram
$$
 \xymatrix{
  D \ar[r]^p & E \\
  X \ar[u]^x \ar[r]^u & Y \ar[u]^y \\
  T \ar[u]^{v'} \ar[r]_q & S \ar[u]^v \ar@/_2pc/[uu]_e
 }
$$
leading to the desired commutative square
diagram~\eqref{slc-QE-epi-class-characterization-diagram}.

 ``If'': given a morphism $s\:U\rarrow S$ in $\sA_{<\kappa}$ and
a morphism $s\rarrow p$ in $\sA^\to$, we need to find a morphism
$r\in\cQ$ such that the morphism $s\rarrow p$ factorizes as
$s\rarrow r\rarrow p$ in~$\sA^\to$.
 So suppose we are given a commutative square diagram
$$
 \xymatrix{
  D \ar[r]^p & E \\
  U \ar[u]_u \ar[r]_s & S \ar[u]_e
 }
$$
in $\sA$ with $\kappa$\+presentable objects $U$ and~$S$.
 By assumption, we can extend the morphisms~$p$ and~$e$ to a commutative
square diagram~\eqref{slc-QE-epi-class-characterization-diagram} with
a morphism $q\in\cQ$.
 Now we have a commutative diagram
$$
 \xymatrix{
  D \ar[r]^p & E \\
  T\sqcup U \ar[u]_{(d,u)}\ar[r]_-{(q,s)} & S \ar[u]_e \\
  U \ar[u]_{i_U} \ar[r]_s \ar@/^2.5pc/[uu]^u
  & S \ar@{=}[u]
 }
$$
where $i_U\:U\rarrow T\sqcup U$ is the coproduct injection.
 Here the coproduct $T\sqcup U$ computed in $\sA$ belongs to
$\sA_{<\kappa}$ by Lemma~\ref{small-colimits-preserved}; so it is
also the coproduct in~$\sA_{<\kappa}$.
 By Proposition~\ref{strong-QE-epi-classes-characterized}(2),
the morphism $r=(q,s)\:T\sqcup U\rarrow S$ belongs to $\cQ$,
providing the desired factorization $s\rarrow r\rarrow p$.
\end{proof}

\begin{thm} \label{strong-QE-epi-preserved-by-correspondence}
 Let\/ $\sA$ be a $\kappa$\+accessible category with finite coproducts.
 Then, for any strong QE\+epi class $\cQ$ in the category\/
$\sA_{<\kappa}$, the (strongly locally $\kappa$\+coherent) QE\+epi
class\/ $\varinjlim_{(\kappa)}\cQ$ is a strong QE\+epi class in
the category\/~$\sA$.
 Conversely, for any strongly locally $\kappa$\+coherent strong
QE\+epi class $\cP$ in\/ $\sA$, the QE\+epi class
$\cP\cap\sA_{<\kappa}^\to$ in the category\/ $\sA_{<\kappa}$ is
a strong QE\+epi class.
\end{thm}

\begin{proof}
 To prove the first assertion, we use the characterization of
the class $\cP=\varinjlim_{(\kappa)}\cQ$ provided by
Proposition~\ref{slc-QE-epi-class-characterization-prop}.
 Let $p'\:C\rarrow D$ and $p''\:D\rarrow E$ be two morphisms in~$\sA$
such that the composition $p''\circ p'$ belongs to~$\cP$.
 Suppose we are given an object $S\in\sA_{<\kappa}$ and a morphism
$e\:S\rarrow E$.
 By Proposition~\ref{slc-QE-epi-class-characterization-prop},
there is a commutative diagram
$$
 \xymatrix{
  C \ar[r]^{p'} & D \ar[r]^{p''} & E \\
  T \ar[u]^c \ar[rr]_q && S \ar[u]_e
 }
$$
with a morphism $q\in\cQ$.
 Setting $d=p'\circ c\:T\rarrow D$, we obtain the desired commutative
diagram~\eqref{slc-QE-epi-class-characterization-diagram} for
the morphisms $p''$ and~$e$.
 Applying Proposition~\ref{slc-QE-epi-class-characterization-prop}
again, we can conclude that $p''\in\cP$.

 The second assertion of the theorem is obvious and does not depend on
the assumption that $\sA$ has finite coproducts.
\end{proof}

\begin{rem}
 One can say that a QE\+mono class $\cM$ in a category $\sC$ is
a \emph{strong QE\+mono class} if it satisfies condition~(iv) dual to
condition~(iv$^*$) from Section~\ref{strong-QE-epi-classes-secn}.
 It would be interesting to know whether a version of
Theorem~\ref{strong-QE-epi-preserved-by-correspondence} holds for
strong QE\+mono classes, or what assumptions on a $\kappa$\+accessible
category $\sA$ are needed for it to hold.
 Let $\cN$ be a strong QE\+mono class in $\sA_{<\kappa}$; does it follow
that $\cM=\varinjlim_{(\kappa)}\cN$ is a strong QE\+mono class in~$\sA$?
\end{rem}

\Section{Regularity of Strongly Pure Monomorphisms}

 Let $\kappa$~be a regular cardinal and $\sA$ be a $\kappa$\+accessible
category.
 The definition of a strongly $\kappa$\+pure monomorphism in $\sA$ was
given in Section~\ref{strongly-pure-monomorphisms-secn}.
 For the definition of a regular monomorphism, see
Section~\ref{QE-mono-classes-secn}.

\begin{prop} \label{strongly-pure-monomorphisms-regular-prop}
 Let\/ $\sA$ be a $\kappa$\+accessible category such that every split
monomorphism in\/ $\sA$ has a cokernel pair.
 Then all strongly $\kappa$\+pure monomorphisms in\/ $\sA$ are
regular monomorphisms.
\end{prop}

\begin{proof}
 The assumption of the proposition means that all split monomorphisms
in $\sA$ are effective (since all split monomorphisms are always
regular; see Section~\ref{QE-mono-classes-secn}).
 By Lemma~\ref{effective-monomorphisms-lemma}, it follows that all
the $\kappa$\+directed colimits of split monomorphisms in $\sA$ are
effective monomorphisms, too.
 For an alternative argument applicable under slightly more restrictive
assumptions, see
Proposition~\ref{strongly-pure-monomorphisms-pushout-stable-prop} below.
\end{proof}

 The definition of a very weak cokernel pair was given in
Section~\ref{very-weak-cokernel-pairs-secn}.

\begin{cor} \label{pure-monomorphisms-regular-cor}
 Let\/ $\sA$ be a $\kappa$\+accessible category with very weak
cokernel pairs (e.~g., this holds if\/ $\sA$ has finite products).
 Assume further that every split monomorphism in\/ $\sA_{<\kappa}$ has
a cokernel pair (in\/ $\sA_{<\kappa}$, or equivalently, in\/~$\sA$).
 Then all $\kappa$\+pure monomorphisms in\/ $\sA$ are regular
monomorphisms.
\end{cor}

\begin{proof}
 In any $\kappa$\+accessible category with very weak cokernel pairs,
the classes of $\kappa$\+pure and strongly $\kappa$\+pure monomorphisms
coincide by Theorem~\ref{under-very-weak-cokernel-pairs-theorem}.
 Any category with finite products has very weak cokernel pairs by
Example~\ref{very-weak-cokernel-pair-examples}(2).
 For a morphism~$i$ in $\sA_{<\kappa}$, the cokernel pair of~$i$ in
$\sA_{<\kappa}$ is the same thing as the cokernel pair of~$i$ in $\sA$
by Lemma~\ref{small-colimits-preserved}.
 The rest is clear from the proof of
Proposition~\ref{strongly-pure-monomorphisms-regular-prop}.
\end{proof}

\begin{ex} \label{in-additive-pure-monomorphisms-are-regular-example}
 Any accessible category has split
idempotents~\cite[Observation~2.4]{AR}.
 In particular, any accessible \emph{additive} category
is idempotent-complete.
 By Remark~\ref{pushouts-of-split-monomorphisms-remark}, it follows
that all pushouts of split monomorphisms exist in any accessible
additive category~$\sA$.
 In particular, all split monomorphisms have cokernel pairs in~$\sA$.
 Thus it follows from Corollary~\ref{pure-monomorphisms-regular-cor}
that all $\kappa$\+pure monomorphisms are regular in any
$\kappa$\+accessible additive category~$\sA$.
 This result also follows from the discussion of
the \emph{$\kappa$\+pure exact structure} on a $\kappa$\+accessible
additive category in~\cite[Section~4]{Plce}.
 The specific references are~\cite[Proposition~2.5]{Plce} (for
the existence of the $\kappa$\+pure exact structure)
and~\cite[Proposition~4.4]{Plce} (for the description of the admissible
monomorphisms in the $\kappa$\+pure exact structure).
\end{ex}

\Section{Pushouts of Strongly Pure Monomorphisms}

 The discussion in this section is just a special case of
Sections~\ref{QE-mono-classes-secn}\+-\ref{QE-mono-construction-secn}.

\begin{prop} \label{strongly-pure-monomorphisms-pushout-stable-prop}
 Let\/ $\sA$ be a $\kappa$\+accessible category such that pushouts
of split monomorphisms exist in\/~$\sA$.
 Then all pushouts of strongly $\kappa$\+pure monomorphisms exist
in\/ $\sA$, and all such pushouts are strongly $\kappa$\+pure
monomorphisms themselves.
\end{prop}

\begin{proof}
 Let $\cN$ be the class of all split monomorphisms in~$\sA_{<\kappa}$.
 For any span in $\sA_{<\kappa}$, if the pushout exists in $\sA$, then
it belongs to $\sA_{<\kappa}$ by Lemma~\ref{small-colimits-preserved}.
 So all pushouts of split monomorphisms exist in $\sA_{<\kappa}$ under
out assumptions.
 By Example~\ref{QE-mono-classes-examples}(3), it follows that
$\cN$ is a QE\+mono class in~$\sA_{<\kappa}$.
 Applying Theorem~\ref{direct-limit-closure-of-QE-mono-theorem},
we conclude that the class of strongly $\kappa$\+pure monomorphisms
$\cM=\varinjlim_{(\kappa)}\cN$ is a QE\+mono class in~$\sA$.
 In particular, condition~(i) from
Section~\ref{QE-mono-classes-secn} is satisfied for $\cM$, as desired.
\end{proof}

\begin{cor} \label{pure-monomorphisms-pushout-stable-cor}
 Let\/ $\sA$ be a $\kappa$\+accessible category with very weak
cokernel pairs (e.~g., this holds if\/ $\sA$ has finite products).
 Assume further that pushouts of split monomorphisms exist
in\/~$\sA_{<\kappa}$.
 Then all pushouts of $\kappa$\+pure monomorphisms exist in\/ $\sA$,
and all such pushouts are $\kappa$\+pure monomorphisms themselves.
\end{cor}

\begin{proof}
 Follows from Theorem~\ref{under-very-weak-cokernel-pairs-theorem},
Example~\ref{very-weak-cokernel-pair-examples}(2), and the proof of
Proposition~\ref{strongly-pure-monomorphisms-pushout-stable-prop}
(cf.\ the proof of Corollary~\ref{pure-monomorphisms-regular-cor}).
\end{proof}

\begin{ex}
 According to
Example~\ref{in-additive-pure-monomorphisms-are-regular-example},
all pushouts of split monomorphisms exist in any accessible
additive category~$\sA$.
 So it follows from
Corollary~\ref{pure-monomorphisms-pushout-stable-cor} that
$\kappa$\+pure monomorphisms are stable under pushouts in
any $\kappa$\+accessible additive category~$\sA$.
 This result also follows from the discussion of the $\kappa$\+pure
exact structure on $\sA$ in~\cite[Section~4]{Plce}; see
Example~\ref{in-additive-pure-monomorphisms-are-regular-example}
for specific references.
\end{ex}

\begin{ex} \label{no-pushouts-of-split-monos-in-acc-preadd-example}
 Any small category with split idempotents is
accessible~\cite[Theorem~2.2.2]{MP}.
 Therefore, Remark~\ref{pushouts-of-split-monomorphisms-remark}
provides an example of an accessible preadditive category that does
not have pushouts (or even cokernel pairs, or even very weak
cokernel pairs) of split monomorphisms.
 So pushouts of pure monomorphisms need not exist in an accessible
preadditive category, generally speaking.
\end{ex}

\Section{Regularity of Strongly Pure Epimorphisms}

 Let $\kappa$~be a regular cardinal and $\sA$ be a $\kappa$\+accessible
category.
 The definition of a strongly $\kappa$\+pure epimorphism in $\sA$ was
given in Section~\ref{strongly-pure-epimorphisms-secn}.
 For the definition of a regular epimorphism, see
Section~\ref{QE-epi-classes-secn}.

\begin{prop} \label{strongly-pure-epimorphisms-regular-prop}
 Let\/ $\sA$ be a $\kappa$\+accessible category such that every split
epimorphism in\/ $\sA$ has a kernel pair.
 Then all strongly $\kappa$\+pure epimorphisms in\/ $\sA$ are
regular epimorphisms.
\end{prop}

\begin{proof}
 The assumption of the proposition means that all split epimorphisms
in $\sA$ are effective (since all split epimorphisms are always
regular; see Section~\ref{QE-epi-classes-secn}).
 By Lemma~\ref{effective-epimorphisms-lemma}(b), it follows that all
the $\kappa$\+directed colimits of split epimorphisms in $\sA$ are
effective epimorphisms, too.
 For an alternative argument applicable under slightly more restrictive
(or slightly different) assumptions, see
Proposition~\ref{strongly-pure-epimorphisms-pullback-stable-prop} below.
\end{proof}

 The definition of a very weak split pullback was given in
Section~\ref{very-weak-split-pullbacks-secn}.

\begin{cor} \label{pure-epimorphisms-regular-cor}
 Let\/ $\sA$ be a $\kappa$\+accessible category with very weak
split pullbacks (e.~g., this holds if\/ $\sA$ has finite coproducts).
 Assume further that every split epimorphism in\/ $\sA_{<\kappa}$ has
a kernel pair in\/~$\sA$ (in particular, this holds if every split
epimorphism has a kernel pair in\/~$\sA_{<\kappa}$).
 Then all $\kappa$\+pure epimorphisms in\/ $\sA$ are regular
epimorphisms.
\end{cor}

\begin{proof}
 In any $\kappa$\+accessible category with very weak split pullbacks,
the classes of $\kappa$\+pure and strongly $\kappa$\+pure epimorphisms
coincide by Theorem~\ref{under-very-weak-split-pullbacks-theorem}.
 Any category with finite coproducts has very weak split pullbacks by
Example~\ref{very-weak-split-pullback-examples}(2).
 For a morphism~$p$ in $\sA_{<\kappa}$, if the kernel pair of~$p$ exists
in $\sA_{<\kappa}$, then it is also the kernel pair of~$p$ in $\sA$
by Lemma~\ref{limits-preserved}.
 The rest is clear from the proof of
Proposition~\ref{strongly-pure-epimorphisms-regular-prop}.
\end{proof}

\begin{ex} \label{in-additive-pure-epimorphisms-are-regular-example}
 Similarly to
Example~\ref{in-additive-pure-monomorphisms-are-regular-example},
all pullbacks of split epimorphisms exist in any accessible
\emph{additive} category $\sA$ by
Remark~\ref{pullbacks-of-split-epimorphisms-remark}.
 In particular, all split epimorphisms have kernel pairs in~$\sA$.
 Similarly, all split epimorphisms have kernel pairs in $\sA_{<\kappa}$,
since $\sA_{<\kappa}$ is idempotent-complete
by Remark~\ref{QE-mono-retraction-closedness-remark}.
 Thus it follows from Corollary~\ref{pure-epimorphisms-regular-cor}
that all $\kappa$\+pure epimorphisms are regular in any
$\kappa$\+accessible additive category~$\sA$.
 This result also follows from the discussion of
the \emph{$\kappa$\+pure exact structure} on a $\kappa$\+accessible
additive category in~\cite[Section~4]{Plce}.
 The specific references are~\cite[Proposition~2.5]{Plce} (for
the existence of the $\kappa$\+pure exact structure)
and~\cite[Proposition~4.2]{Plce} (for the description of the admissible
epimorphisms in the $\kappa$\+pure exact structure).
\end{ex}

\Section{Pullbacks of Strongly Pure Epimorphisms}

\begin{prop} \label{strongly-pure-epimorphisms-pullback-stable-prop}
 Let\/ $\sA$ be a $\kappa$\+accessible category such that pullbacks
of split epimorphisms exist in the category\/~$\sA_{<\kappa}$.
 Then all pullbacks of strongly $\kappa$\+pure epimorphisms exist
in\/ $\sA$, and all such pullbacks are strongly $\kappa$\+pure
epimorphisms themselves.
\end{prop}

\begin{proof}
 Let $\cQ$ be the class of all split epimorphisms in~$\sA_{<\kappa}$.
 By Example~\ref{QE-epi-classes-examples}(3),
\,$\cQ$ is a QE\+epi class in~$\sA_{<\kappa}$.
 Applying Theorem~\ref{direct-limit-closure-of-QE-epi-theorem},
we conclude that the class of strongly $\kappa$\+pure epimorphisms
$\cP=\varinjlim_{(\kappa)}\cQ$ is a QE\+epi class in~$\sA$.
 In particular, condition~(i$^*$) from
Section~\ref{QE-epi-classes-secn} is satisfied for $\cP$, as desired.
\end{proof}

\begin{cor} \label{pure-epimorphisms-pullback-stable-cor}
 Let\/ $\sA$ be a $\kappa$\+accessible category with very weak
split pullbacks (e.~g., this holds if\/ $\sA$ has finite coproducts).
 Assume further that pullbacks of split epimorphisms exist
in\/~$\sA_{<\kappa}$.
 (More generally, it suffices to assume that every split epimorphism
in\/ $\sA_{<\kappa}$ has a pullback in\/ $\sA$ along every morphism
in\/~$\sA_{<\kappa}$.)
 Then all pullbacks of $\kappa$\+pure epimorphisms exist in\/ $\sA$,
and all such pullbacks are $\kappa$\+pure epimorphisms themselves.
\end{cor}

\begin{proof}
 The main assertion is obtained by combining
Theorem~\ref{under-very-weak-split-pullbacks-theorem},
Example~\ref{very-weak-split-pullback-examples}(2), and
Proposition~\ref{strongly-pure-epimorphisms-pullback-stable-prop}
(cf.\ the proof of Corollary~\ref{pure-epimorphisms-regular-cor}).
 For the more general assertion under the assumption in parentheses,
one needs to follow the proof of condition~(i$^*$) for the class $\cP$
in Theorem~\ref{direct-limit-closure-of-QE-epi-theorem}.
\end{proof}

\begin{ex}
 According to
Example~\ref{in-additive-pure-epimorphisms-are-regular-example},
all pullbacks of split epimorphisms exist in any $\kappa$\+accessible
additive category $\sA$, as well as in its full
subcategory~$\sA_{<\kappa}$.
 So it follows from
Corollary~\ref{pure-epimorphisms-pullback-stable-cor} that
$\kappa$\+pure epimorphisms are stable under pullbacks in
any $\kappa$\+accessible additive category~$\sA$.
 This result also follows from the discussion of the $\kappa$\+pure
exact structure on $\sA$ in~\cite[Section~4]{Plce}; see
Example~\ref{in-additive-pure-epimorphisms-are-regular-example}
for specific references.
\end{ex}

\begin{ex} \label{no-pullbacks-of-split-epis-in-acc-preadd-example}
 Any small category with split idempotents is
accessible~\cite[Theorem~2.2.2]{MP}.
 Therefore, Remark~\ref{pullbacks-of-split-epimorphisms-remark}
provides an example of an accessible preadditive category that does
not have pullbacks (or even kernel pairs) of split epimorphisms.
 So pullbacks of pure epimorphisms need not exist in an accessible
preadditive category in general.
\end{ex}

\end{document}